\documentclass[a4paper,11pt]{article}
\usepackage{amsfonts}
\usepackage{amscd,color}
\usepackage{amsmath,amsfonts,amssymb,amscd}
\usepackage{indentfirst,graphicx,epsfig}
\usepackage{graphicx,psfrag}
\input{epsf}

\newtheorem{df}{Definition}[section]
\newtheorem{thm}{Theorem}[section]
\newtheorem{prop}[thm]{Proposition}

\newtheorem{lem}[thm]{Lemma}
\newtheorem{ob}{Observation}[section]
\newtheorem{claim}{Claim}[section]

\newenvironment {proof} {\noindent{\em Proof.}}{\hspace*{\fill}$\Box$\par\vspace{4mm}}

\def\qed{\hfill \nopagebreak\rule{5pt}{8pt}}

\title{Anti-Ramsey numbers of paths and cycles in hypergraphs}

\author{
\small  Ran Gu$^1$, Jiaao Li$^2$ and  Yongtang Shi$^3$\\
\small $^1$College of Science, Hohai University,\\
\small Nanjing, Jiangsu Province 210098,
P.R. China\\
\small $^2$School of Mathematical Sciences and LPMC \\
\small Nankai University, Tianjin 300071, China \\
\small $^3$Center for Combinatorics and LPMC \\
\small Nankai University, Tianjin 300071, China \\
\small Emails: rangu@hhu.edu.cn; lijiaao@nankai.edu.cn; shi@nankai.edu.cn
\\
\date{}}

\begin{document}
\maketitle
\begin{abstract}
The anti-Ramsey problem was introduced by Erd\H os, Simonovits and S\'{o}s in 1970s. The anti-Ramsey number of a  hypergraph $\mathcal{H}$, $ar(n,s, \mathcal{H})$, is the smallest integer $c$ such that in any coloring of the
edges of the $s$-uniform complete hypergraph on $n$ vertices  with exactly $c$ colors, there is a copy of $\mathcal{H}$ whose edges have distinct colors.  In this paper, we determine the anti-Ramsey numbers of  linear paths and loose paths in hypergraphs for sufficiently large $n$, and give bounds for the anti-Ramsey numbers of  Berge paths. Similar exact anti-Ramsey numbers are obtained for linear/loose cycles, and bounds are obtained for Berge cycles. Our  main tools are path extension technique and stability results on hypergraph Tur\'{a}n problems of paths and cycles.
\\[2mm]
[2mm] \textbf{AMS Subject Classification (2010):} 05D05, 05C35, 05C65
\end{abstract}

\section{Introduction}
The
\emph{anti-Ramsey number} of a graph $G$, denoted by $ar(n,G)$, is the minimum number of colors needed to color the edges of the complete graph $K_n$ so
that, in any coloring, there exists a  copy of $G$ whose edges have distinct colors. The \emph{Tur\'{a}n number} of a graph $G$, denoted by $ex(n, G)$, is the maximum
number of edges in a graph on $n$ vertices that does not contain $G$ as a subgraph.  It is easy to observe that
\begin{equation}\label{H-e}
2+ex(n,\{H-e, e\in E(H)\})\leq ar (n,H)\leq ex(n,H)+1,
\end{equation}
for any graph $H$.

In 1973, Erd\H os, Simonovits and S\'{o}s \cite{ESS} showed a remarkable result that $ar(n,K_p) = ex(n, K_{p-1}) + 2$ for sufficiently large $n$.
Montellano-Ballesteros and Neumann-Lara \cite{MN} extended this result to all values of $n$ and $p$ with $n > p \ge 3$. In \cite{ESS}, it was shown that $ar (n,H)-ex(n,\{H-e, e\in E(H)\})=o(n^2)$ when $n\rightarrow \infty$. Furthermore, Jiang \cite{Jiang} proved that if $H$ is a graph such that each edge is incident to a
vertex of degree two, then $ar (n,H)-ex(n,\{H-e, e\in E(H)\})=O(n)$. A history
of results and open problems on this topic was given by Fujita, Magnant, and Ozeki \cite{FMO}.

A \emph{hypergraph} $\mathcal{H}$ consists of a set $V(\mathcal{H})$ of vertices and a family $E(\mathcal{H})$ of nonempty subsets of $V(\mathcal{H})$ called edges of $\mathcal{H}$. If
each edge of $\mathcal{H}$ has exactly $s$ vertices, then $\mathcal{H}$ is {\em $s$-uniform} and $\mathcal{H}$ is called an {\em $s$-graph}. A \emph{complete $s$-uniform hypergraph} is a hypergraph whose edge
set consists of all $s$-subsets of the vertex set. In an edge-coloring of a (hyper)graph $\mathcal{H}$, a sub(hyper)graph $\mathcal{F}\subseteq\mathcal{H}$ is \emph{rainbow} if all edges of $\mathcal{F}$ have distinct colors.

The anti-Ramsey number and Tur\'{a}n number are naturally extended from  graphs to hypergraphs. The anti-Ramsey number of an $s$-uniform hypergraph $\mathcal{H}$,
denoted by $ar(n, s, \mathcal{H})$, is the minimum number of colors needed to color the edges of a complete $s$-uniform hypergraph on $n$
vertices so that there exists a rainbow $\mathcal{H}$ in any coloring. The Tur\'{a}n number of $\mathcal{H}$, denoted by $ex(n, s, \mathcal{H})$,
is the maximum number of edges in an $s$-uniform hypergraph on $n$ vertices that contains no $\mathcal{H}$.  \"{O}zkahya and Young \cite{OY} investigated the anti-Ramsey number of matchings in hypergraphs, where a \emph{matching} is a set of edges in a (hyper)graph in which no two edges have
a common vertex.   A $k$-\emph{matching}, denoted by $M_k$, is a matching with $k$ edges. \"{O}zkahya and Young \cite{OY} gave the lower and upper bounds for $ar(n, s, M_k)$  in terms of $ex(n, s, M_{k-1})$. They proved that $$ ex(n, s, M_{k-1}) + 2\le ar(n, s, M_k)\le ex(n, s, M_{k-1}) + k,$$ where the lower bound holds for every $n$, and the upper bound holds for $n\ge  sk + (s - 1)(k - 1)$. For $s=2$, Schiermeyer \cite{Schiermeyer} proved that $ar(n, 2, M_{k}) = ex(n, 2, M_{k-1}) + 2$ for $k \geq 2$ and $n\geq 3k + 3$, and this condition was further released to all $n\geq 2k+1$ by Chen, Li and Tu \cite{CLT2009} and  by Fujita, Kaneko, Schiermeyer and Suzuki \cite{FKSS2009}, independently.

In fact, for $k$-matchings, the Tur\'{a}n number $ex(n, s, M_k)$ is still not known for $k \ge 3$ and $s \ge 3$. Let $[n]$ denote the set $\{1,2,\ldots,n\}$, $\binom{[n]}{k}$ denote the set consisting of all the $k$-sets of $[n]$.
 Erd\H{o}s put forward a conjecture in 1965  that $ex(n,s,M_{k})=\max\{|A_s|,|B_s(n)|\},$ where
$A_s={[sk-1] \choose s}$ and $B_s(n)=\{F\in {[n]\choose s}|
F\cap [k-1]\neq \emptyset\}$.
This conjecture is true for $s = 2$, which was
shown by Erd\H{o}s and Gallai \cite{EG}. In \cite{N14}, Erd\H{o}s proved that there exists a constant $n_0(s,k)$ such that for $n>n_0(s,k)$, the conjecture holds. Then Bollob\'as,  Daykin and  Erd\H{o}s \cite{N5} improved the bound for $n_0(k,s)$ such that $n_0(k,s)\le 2s^3(k-1)$. It was  improved to $n_0(k,s)\le 3s^2(k-1)$ by Huang, Loh and Sudakov \cite{N21} later.

For the anti-Ramsey number of $k$-matching, \"{O}zkahya and Young \cite{OY}  conjectured
that when $k \ge 3$, $ar(n, s, M_k) = ex(n, s, M_{k-1}) + 2$ if $n > sk$, and
  \begin{equation*}
ar(n, s, M_k)=
\left\{
  \begin{array}{ll}
  ex(n, s, M_{k-1}) + 2 & \hbox{  if $k \le c_s$,} \\
  ex(n, s, M_{k-1}) + s+1 & \hbox{  if $k\ge c_s$},
  \end{array}
\right.
\end{equation*}
if $n = sk$,  where $c_s$ is a constant dependent on $s$. They proved that the conjecture is true when $k = 2,3$ for sufficiently large $n$. Later, Frankl and Kupavskii \cite{FK}  proved that $ar(n, s, M_k) =
ex(n, s, M_{k-1}) + 2$ for $n \ge sk+ (s-1)(k-1)$
and $k \ge3$. For more results on matchings, we refer to \cite{frankl1, FK1}.

{
 For paths, Simonovits and S\'{o}s \cite{SS} proved that  $ar(n, P_{2t+3+\epsilon})=tn-\binom{t-1}{2}+1+\epsilon$ for large $n$, where $\epsilon=0, 1$ and $P_k$ is a
path on $k$ vertices. Comparing with the Tur\'an number of paths
\begin{equation}\label{eqP}
ex(n, P_{k} )\le (k-2)n/2
\end{equation}
 given by Erd\H{o}s and Gallai \cite{EG}, it follows that $ar(n, P_{k})=ex(n,P_{k-1})+O(1)$ when $k$ is odd, and $ar(n, P_{k})=ex(n,P_{k-2})+O(1)$ when $k$ is even. For a cycle $C_k$ of order $k$, Erd\H{o}s,  Simonovits and S\'{o}s \cite{ESS} conjectured that
  $ar(n, C_{k})=n\left(\frac{k-2}{2}+\frac{1}{k-1}\right)+O(1)$. This conjecture was confirmed by Montellano-Ballesteros and Neumann-Lara \cite{MBNL}, and they gave the exact value of $ar(n, C_{k})$ for all $n\ge k\ge 3$.
It would be interesting to investigate the relation between the anti-Ramsey number and the Tur\'an number for paths and cycles in hypergraphs. The Tur\'an numbers of paths  and cycles are extensively studied, see \cite{frankl,FJ,FJS,KMV} or  Section 2 below for details. Motivated by this, we will study the anti-Ramsey numbers of paths and cycles and compare it with the Tur\'an numbers of paths  and cycles in hypergraphs.
}

There are several possible ways to define paths and cycles in hypergraphs as generalization of paths and cycles in graphs from different aspects.

\begin{df}
Let $\mathcal{H}$ be an $s$-uniform hypergraph.

 (i) A {\bf Berge path} of length $k$ in $\mathcal{H}$ is a family of $k$
distinct edges $e_1,\ldots , e_k$ and $k+1$
distinct vertices $v_1, \ldots, v_{k+1}$ such that for each
$1 \leq i \leq k$, $e_i$ contains $v_i$ and $v_{i+1}$. Let
$\mathcal{B}_k$ denote the family of  Berge
paths of length $k$. A {\bf Berge cycle} of length $k$ in $\mathcal{H}$ is a cyclic list of $k$ distinct edges $e_1,\ldots , e_k$ and $k$ distinct vertices $v_1, \ldots, v_k$ such that $e_i$ contains $v_i$ and $v_{i+1}$ for each $1 \leq i \leq k$, where $v_{k+1} = v_1$. Denote the family of all Berge cycles of
length $k$ by $\mathcal{BC}_{k}$.


(ii) A {\bf loose path} of length $k$ in $\mathcal{H}$
is a collection of  distinct edges
$\{e_1, e_2, \ldots, e_{k}\}$  such
that consecutive edges intersect in at least one element and
nonconsecutive edges are disjoint. Denote the family of loose paths of
length $k$ by $\mathcal{P}_{k}$. A  {\bf loose cycle} is defined similarly in a cyclic order, and denote the family of all loose cycles of
length $k$ by $\mathcal{C}_{k}$.

(iii) A {\bf linear path} of length $k$ in $\mathcal{H}$
is a collection of  distinct edges
$\{e_1, e_2, \ldots, e_{k}\}$  such
that consecutive edges intersect in exactly one element and
nonconsecutive edges are disjoint. Let $\mathbb{P}_{k}$ denote the  linear
path of length $k$.  A
 {\bf linear cycle} is defined similarly in a cyclic order, and let  $\mathbb{C}_{k}$ denote the collection of  linear
path of length $k$.
\end{df}


We first give the exact anti-Ramsey numbers of short paths $\mathbb{P}_{i}$, $\mathcal{B}_i$, $\mathcal{P}_i$ for $i=2,3$.
\begin{thm}\label{P2}
(i) For $s\ge 3$ and $n\ge 3s-4$, $ar(n,s,\mathbb{P}_{2})=2$.

(ii) For $s\ge 4$ and sufficiently large $n$, $ar(n,s,\mathbb{P}_{3})=\binom{n-2}{s-2}+2$.

(iii) For $n\ge 3s-4$, $ar(n,s,\mathcal{B}_{2})= ar(n,s,\mathcal{P}_{2})=2$.

(iv) For $n\ge 4s-3$, $ar(n,s,\mathcal{B}_{3})=ar(n,s,\mathcal{P}_{3})=3$.
\end{thm}

For linear paths and loose paths, we obtain the exact anti-Ramsey numbers for sufficiently large $n$.
\begin{thm}\label{th1}
For any integer $k$, if $k=2t\ge4$ and $s\ge 3$, then for  sufficiently large $n$,
$$ ar(n,s,\mathbb{P}_{k})={ \binom{n}{s}-\binom{n-t+1}{s}+2 };$$
if $k=2t+1>5$ and $s\ge 4$, then for  sufficiently large $n$,  $$ar(n,s,\mathbb{P}_{k})=\binom{n}{s}-\binom{n-t+1}{s}+ \binom{n-t-1}{s-2}+2.$$
\end{thm}

\begin{thm}\label{th2}
For any integer $k$, if $k=2t\ge 4$ and $s\ge 3$, then for sufficiently large $n$, $$ar(n,s,\mathcal{P}_{k})=\binom{n}{s}-\binom{n-t+1}{s}+2;$$
if $k=2t+1\ge 5$ and $s\ge 3$, then for  sufficiently large $n$,  $$\ ar(n,s,\mathcal{P}_{k})=\binom{n}{s}-\binom{n-t+1}{s}+3.$$

\end{thm}

{We remark that, due to some technique obstruction, our proof of Theorem \ref{th1} does not work directly for the case $k=5$ or the case  $k$ is odd and $s=3$. However,  those special cases are handled in Theorem  \ref{th2} for loose path with a refined analysis.
}

The methods developed in proving Theorems \ref{th1} and \ref{th2}, with additional effort and some new ideas, allow us to determine the anti-Ramsey numbers of linear cycles and loose cycles as well, if $k$ and $s$ are not too small. We obtain the following exact results for linear cycles and loose cycles.
\begin{thm}\label{thc1}
For any integer $k$, if $k=2t\ge8$ and $s\ge 4$, then for  sufficiently large $n$,
$$ ar(n,s,\mathbb{C}_{k})=ar(n,s,\mathbb{P}_{k})={ \binom{n}{s}-\binom{n-t+1}{s}+2 };$$
if $k=2t+1\ge 11$ and $s\ge k+3$, then for  sufficiently large $n$,  $$ar(n,s,\mathbb{C}_{k})=ar(n,s,\mathbb{P}_{k})=\binom{n}{s}-\binom{n-t+1}{s}+ \binom{n-t-1}{s-2}+2.$$
\end{thm}

\begin{thm}\label{thc2}
For any integer $k$, if $k=2t\ge 8$ and $s\ge 4$, then for sufficiently large $n$, $$ar(n,s,\mathcal{C}_{k})=ar(n,s,\mathcal{P}_{k})=\binom{n}{s}-\binom{n-t+1}{s}+2;$$
if $k=2t+1\ge 11$ and $s\ge k+3$, then for  sufficiently large $n$,  $$\ ar(n,s,\mathcal{C}_{k})=ar(n,s,\mathcal{P}_{k})=\binom{n}{s}-\binom{n-t+1}{s}+3.$$
\end{thm}

~

For a Berge path $\mathcal{B}_k$, Gy\"{o}ri, Katona and Lemons in \cite{GKL} proved that $ex(n,s,\mathcal{B}_{k})\le \frac{n}{k}{k \choose s}$ when $k>s+1>3$, and $ex(n,s,\mathcal{B}_{k})\le \frac{n(k-1)}{s+1}$ when $2<k\le s$, which are sharp for infinitely many $n$. Then  Davoodi, Gy\"{o}ri, Methuku and Tompkins \cite{DGMT} proved that $ex(n,s,\mathcal{B}_{s+1})\le n$.  We apply those results to  obtain the bounds for the anti-Ramsey number $ar(n,s,\mathcal{B}_{k})$ as follows.

\begin{thm}\label{th3}
If $k>2s+1$,  then for sufficiently large $n$,
  $$\frac{2n}{k}{\lfloor k/2\rfloor \choose s}\le ar(n,s,\mathcal{B}_{k})\le\frac{n}{k-1}\binom{k-1}{s}+1.$$
If $s+2\le k\le2s+1$,  then for sufficiently large $n$,
$$\frac{n}{s+1}\lfloor \frac{k-2}{2}\rfloor \le ar(n,s,\mathcal{B}_{k})\le \frac{n}{k-1}\binom{k-1}{s}+1.$$
If $ k\le s+1$,  then for sufficiently large $n$,  $$\frac{n}{s+1}\lfloor \frac{k-2}{2}\rfloor\le ar(n,s,\mathcal{B}_{k})\le \frac{(k-2)n}{s+1}+1.$$
\end{thm}

Theorem \ref{th3} indicates that the anti-Ramsey number $ar(n,s,\mathcal{B}_{k})$ varies for different  $s$ and $k$.  This may suggest that determining the exact value of $ar(n,s,\mathcal{B}_{k})$ would be very difficult. Note that  the Tur\'an number of Berge path is still not clear at this moment.

However, it seems that  the anti-Ramsey numbers of  Berge cycles have different behavior with Berge paths, and we obtain the following bounds, similar to the \"{O}zkahya-Young result \cite{OY} on matchings.
\begin{prop}\label{BC}
For any fixed integers  $s\ge 4$, $k\ge 3$,
$$ex(n,s,\mathcal{B}_{k-1})+2\le ar(n,s,\mathcal{BC}_{k-1})\le ex(n,s,\mathcal{B}_{k-1})+k.$$
\end{prop}

The next section will be focused on introducing results on Tur\'an numbers of paths and cycles in hypergraphs, which are  needed tools to derive our main results. A useful lemma obtained from the stability results on Tur\'{a}n problems will be given in the next section as well, which will be frequently used to find certain desired paths in later proofs.    The proof of the main results will be presented in later sections.

\section{Preliminaries}
Note that the $s$-uniform Berge path $\mathcal{B}_2$ and loose path $\mathcal{P}_2$ are the same definitions, and the determination of $ex(n,s,\mathcal{P}_2)$ is  trivial.
In \cite{FJS}, F\"{u}redi, Jiang and  Seiver determined  $ex(n,s,\mathcal{P}_k)$ for $s\ge 3$.
\begin{thm}\cite{FJS}\label{turanloose}
Let $s$, $t$ be positive integers with $s\geq 3$. For sufficiently
large $n$, we have
\[ex\left( n,s,\mathcal{P}_{2t+1} \right) =\binom{n}{s}-
\binom{n-t}{s}\]
and \[ex\left( n,s,\mathcal{P}_{2t+2}\right) =\binom{n}{s}-
\binom{n-t}{s}+1.\]
For $\mathcal{P}_{2t+1}$, the unique extremal family  consists of all the $s$-subsets of $[n]$ which meet some fixed set $S$ of size $t$. For $\mathcal{P}_{2t+2}$, the unique extremal family  consists of all the above edges plus one additional $s$-set  disjoint from $S$.
\end{thm}

 The determination of $ex(n,s, \mathbb{P}_k)$ is nontrivial even for $k=2$.
 Frankl \cite{frankl} gave the value of $ex(n,s, \mathbb{P}_2)$
 for $s\ge 4$ and sufficiently large $n$.  Then Keevash, Mubayi and Wilson \cite{KMW} determined  the value of $ex(n,4, \mathbb{P}_2)$ for all $n$. Note
that when $s = 3$, $ex(n,3, \mathbb{P}_2)\le n$, which can be achieved when $n$ is divisible
by 4 by taking $n/4$ vertex disjoint copies of $K^{(3)}_
4$ (i.e. the complete 3-graph on
4 vertices). For $k\ge 3$, F\"{u}redi, Jiang and Seiver \cite{FJS} provided the exact Tur\'an number of  $\mathbb{P}_{k}$ for sufficiently large $n$, where $s\ge4$, $k\ge 3$. Kostochka,  Mubayi and Verstra\"ete \cite{KMV} considered $ex(n,s, \mathbb{P}_k)$ for $s\ge 3$, $k\ge 4$ and sufficiently large $n$. Later, Jackowska, Polcyn and  Ruci\'{n}ski \cite{JPR} determined $ex(n,3, \mathbb{P}_3)$ for all $n$.  We summarize those results (only for the sufficiently large $n$) as follows.

\begin{thm}\label{turan}\cite{frankl,FJS,JPR,KMV,KMW}
For sufficiently
large $n$, we have
\begin{enumerate}
  \item $ex\left( n,s,\mathbb{P}_{2} \right) =\binom{n-2}{s-2}$ for $s\ge 4$, and $ex\left( n,3,\mathbb{P}_{2} \right) \le n$.
      \item $ex\left( n,s,\mathbb{P}_{2t+1} \right) =\binom{n}{s}-
\binom{n-t}{s}$ for $s\ge 3$ and $t\ge 1$.
  \item $ex\left( n,s,\mathbb{P}_{2t+2} \right) =\binom{n}{s}-
\binom{n-t}{s}+\binom{n-t-2}{s-2}$ for $s\ge 3$ and $t\ge 1$.
\end{enumerate}
The unique extremal family for $\mathbb{P}_{2}$ consists of all the $s$-subsets of $[n]$ containing some two fixed vertices for $s\ge 4$.
For $\mathbb{P}_{2t+1}$, the unique extremal family consists of all the $s$-subsets of $[n]$ which meet some fixed set $S$ of size $t$. For $\mathbb{P}_{2t+2}$, the unique extremal family consists of all the above edges plus all the  $s$-sets in $[n]\backslash S$ containing some two fixed vertices not in  $S$.
\end{thm}

For linear cycles, Frankl and F\"{u}redi  \cite{FF} showed that
the unique extremal $n$-vertex $s$-graph ($s\ge 3$)  containing no $\mathbb{C}_{3}$  consists of all edges containing
some fixed vertex $x $, for large enough $n$. For  $s = 3$, Cs\'{a}k\'{a}ny and Kahn \cite{CK} obtained the same result
for all $n \ge 6$.  F\"{u}redi and Jiang \cite{FJ}, and Kostochka,  Mubayi and Verstra\"ete \cite{KMV} determined the Tur\'an number of  $\mathbb{C}_{k}$ for all $k \ge 3$, $s \ge 3$ and sufficiently large $n$ as follows.
\begin{thm}\cite{FJ,KMV}\label{hFJ2}
Let $s,t$ be  positive integers with $s\geq 3$.  For sufficiently large $n$, we have
$$e{x}\left( n,s,\mathbb{C}_{2t+1} \right) = \binom{n}{s}-
\binom{n-t}{s} $$ and for $(s,t)\neq (3,1)$, $$e{x}\left( n,s, \mathbb{C}_{2t+2} \right) =
\binom{n}{s}-
\binom{n-t}{s}+\binom{n-t-2}{s-2}.$$
For $\mathbb{C}_{2t+1}$, the only extremal family
consists of all the $s$-sets in $[n]$ that meet some fixed $t$-set
$L$. For $\mathbb{C}_{2t+2}$, the only extremal
family consists of all the $s$-sets in $[n]$ that intersect some
fixed $t$-set $L$ plus all the $s$-sets in $[n]\setminus L$ that
contain some two fixed elements.
\end{thm}

For the exceptional case of $\mathbb{C}_{4}$ in $3$-uniform hypergraphs,
Kostochka,  Mubayi and Verstra\"ete \cite{KMV} showed that
$$e{x}\left( n,3,\mathbb{C}_{4} \right) = \binom{n}{s}-\binom{n-1}{s}+\max\left\{n-3,4\left\lfloor
\frac{n-1}{4}\right\rfloor\right\},$$
and they also characterized the extremal graphs.

The Tur\'an number of a loose cycle was initially studied by Chv\'atal \cite{Chvatal}. Then Mubayi and Verstra\"ete \cite{MV2005} proved that $e{x}\left( n,s,\mathcal{C}_{3} \right) = \binom{n}{s}-\binom{n-1}{s-1}$ for all $s\ge 3$ and $n\ge 3s/2$.  F\"{u}redi and Jiang \cite{FJ}
determined $e{x}\left( n,s,\mathcal{C}_{k} \right) $ for $k\ge3$, $s\ge 4$  and sufficiently large $n$. This confirms (in a
stronger form) a conjecture proposed by Mubayi and Verstra\"ete \cite{MV} for
$k\ge3$, $s\ge 4$. Kostochka,  Mubayi and Verstra\"ete \cite{KMV} extended the results above  and determined $e{x}\left( n,s,\mathcal{C}_{k} \right) $  for all
$s \ge3 $ and large $n$. We summarize their  results as follows.
\begin{thm}\cite{FJ,KMV}\label{hFJ1}
Let $t\ge 2$, $s\geq3$ be fixed integers. For sufficiently large $n$, we have
$$e{x}\left( n,s,\mathcal{C}_{2t+1} \right) = \binom{n}{s}-
\binom{n-t}{s},$$  $$e{x}\left( n,s,\mathcal{C}_{2t+2}
\right) = \binom{n}{s}-
\binom{n-t}{s}+1,$$
and
$$e{x}\left( n,s,\mathcal{C}_{4}
\right) = \binom{n}{s}-\binom{n-1}{s}+\left\lfloor\frac{n-1}{s}
\right\rfloor.$$
For $\mathcal{C}_{2t+1}$ ($t\ge 2$), the only extremal
family consists of all the $s$-sets in $[n]$ that meet some fixed
$t$-set $L$. For $\mathcal{C}_{2t+2}$  ($t\ge 2$), the only
extremal family consists of all the $s$-sets in $[n]$ that intersect
some fixed $t$-set $L$ plus one additional $s$-set outside $L$. For $\mathcal{C}_{4}$, the only
extremal family consists of all the $s$-sets in $[n]$ that intersect
some fixed $t$-set $L$ plus $\left\lfloor\frac{n-1}{s}
\right\rfloor$ disjoint edges outside $L$.
\end{thm}

For Tur\'an number of a Berge cycle, Gy\H{o}ri and Lemons  proved that $ex(n,3,\mathcal{BC}_{2k+1}) \le O(k^4)\cdot  n^{1+1/k}$ in \cite{GL3}
and  $ex(n,3,\mathcal{BC}_{2k+1}) \le O(k^2) \cdot ex(n,C_{2k})$ in \cite{GL}.  F\"{u}redi and \"{O}zkahya \cite{FO} obtained better constant factors (depending on $k$).  Further improvements  were obtained for even $k$ by Gerbner, Methuku and Vizer \cite{GMV}, also by
Gerbner, Methuku and Palmer \cite{GMP}, and for odd $k$ by Gerbner \cite{Gerbner}.
 For $s\ge 4$, Gy\H{o}ri and Lemons \cite{GL} showed that $ex(n,s,\mathcal{BC}_{2k+1})\le O_s(k^{s-2})\cdot ex(n,3,\mathcal{BC}_{2k+1})$ and  $ex(n,s,\mathcal{BC}_{2k})\le O_s(k^{s-1})\cdot ex(n,C_{2k})$, i.e., $ ex(n,s,\mathcal{BC}_{k})= O(n^{1+1/\lfloor k/2\rfloor} )$. The constant factors were improved
by Jiang and Ma \cite{JM}, also for even $k$ by Gerbner, Methuku and Vizer \cite{GMV}.

There are many other results on the Tur\'{a}n numbers of paths and cycles,  in graphs \cite{BJ,BK2011,LLP,YZ} or  hypergraphs \cite{BK,FO,GKL}.  The readers are referred to these references for details.

The following stability result on linear paths and linear cycles will be needed in our proofs. Let $\partial \mathcal{H}$ denote
the $(s- 1)$-graph consisting of sets contained  in some edge of $\mathcal{H}$.
\begin{thm}\label{stability}\cite{KMV}
For fixed  $s
\ge 3$ and  $ k\ge 4$, let $\ell=\left\lfloor\frac{k-1}{2}\right\rfloor$ and  $\mathcal{H}$ be  an  $n$-vertex  $s$-graph  with  $|\mathcal{H}| \sim \ell\binom{n}{s-1}$
containing  no  $\mathbb{P}_k$ or containing no $\mathbb{C}_k$.  Then  there  exists  $G^*\subset \partial \mathcal{H}$ with  $|G^*|  \sim \binom{n}{s-1}$ and a set $L$ of $\ell$ vertices of $\mathcal{H}$ such that $L\cap V(G^*)=\emptyset$ and $e\cup \{v\}\in \mathcal{H}$ for any $(s-1)$-edge $e\in G^*$  and any $v\in L$.  In  particular,
$|\mathcal{H }- L|= o(n^{s-1})$.
\end{thm}

Notice that the stability result above considers the case $ k\ge 4$. For $k=3$, $s\ge4$,  F\"{u}redi et al. \cite{FJS} provided another version of stability result on linear paths, and as the authors in \cite{JPR} pointed out, the similar stability result holds for $k=3$ and $s=3$ as well. We rewrite their results for $k=3$ with the similar notations in  Theorem \ref{stability} (in a slightly weaker form). Note that when $k=3$, $\ell=\left\lfloor\frac{k-1}{2}\right\rfloor=1$.
\begin{thm}\label{stability2}
For fixed   $s
\ge 3$,  let   $\mathcal{H}$ be  an  $n$-vertex  $s$-graph  with  $|\mathcal{H}| \sim \binom{n}{s-1}$
containing  no  $\mathbb{P}_3$.  Then  there  exists  $G^*\subset \partial \mathcal{H}$ with  $|G^*|  >\frac{1}{2} \binom{n}{s-1}$ and a vertex $v$  of $\mathcal{H}$ such that $v\notin V(G^*)$ and $e\cup \{v\}\in \mathcal{H}$ for any $(s-1)$-edge $e\in G^*$.  In  particular,
$|\mathcal{H }- v|= o(n^{s-1})$.
\end{thm}

Considering the structure of the $(s-1)$-graph $G^*$, we present
the following lemma, which is frequently  used in our proofs.
\begin{lem}\label{lem1}
For fixed  $s
\ge 3$ and  $ k\ge 3$, let $t=\left\lfloor\frac{k-1}{2}\right\rfloor$ and  $\mathcal{H}$ be  an  $n$-vertex  $s$-graph  with  $|\mathcal{H}| \sim t\binom{n}{s-1}$, which
contains  no  $\mathbb{P}_k$, or contains no $\mathbb{C}_k$ when $k\ge 4$. Let $G^*\subset \partial \mathcal{H}$ be the $(s-1)$-graph as defined in Theorem \ref{stability} or \ref{stability2} above.
Given a vertex set $W$ of $d$ vertices in $\mathcal{H}$, where $d$ is a fixed constant.
Then for sufficiently large $n$, there are $\max\{t-1,1\}$ pairs of $(s-1)$-edges in $G^*$,
say $\{a_i,b_i\}$, $i=1,\ldots,t-1$,
such that each of the following holds.

 (i) For every $i$, $a_i$ and $b_i$ have exactly one common vertex (i.e. $|a_i\cap b_i|=1$),

 (ii) for any $i\neq j$,
$a_i\cup b_i$ and $a_j\cup b_j$ are vertex disjoint, and moreover,

 (iii) all these $(s-1)$-edges are disjoint from $W$.
\end{lem}
\begin{proof}
The number of $(s-1)$-edges incident with some vertices in $W$ is at
most $|W|\cdot \binom{n - 1}{s - 2}$, so in $G^*$ the number of $(s-1)$-edges disjoint from $W$ is
at least
\begin{equation*}
|G^*|-d\binom{n - 1}{s - 2}>
\left\{
  \begin{array}{ll}
  \binom{n}{s-3}& \hbox{ for $s\ge 4$,} \\[3mm]
  n  & \hbox{  for $s=3$}.
  \end{array}
\right.
\end{equation*}
By Theorem \ref{turan} and Eq.\eqref{eqP}, we get that  $|G^*|-d\binom{n - 1}{s - 2}>ex(n,s-1,\mathbb{P}_2)$  for sufficiently large $n$. So we can find a pair $\{a_1,b_1\}$ of
$(s-1)$-edges with exactly one common vertex.
Since  $$|G^*|-d\binom{n - 1}{s - 2}-(t-1)(2s-3)\binom{n - 1}{s - 2}
> ex(n,s-1,\mathbb{P}_2),$$
we can repeat the argument above to find $\{a_2,b_2\}$,
$\ldots,\{a_{t-1},b_{t-1}\}$ satisfying the properties described in Lemma
\ref{lem1}.
\end{proof}

Given a path $P$, if a vertex $v$  belongs to more than one edge in $P$, we call $v$ a \emph{cross vertex} of $P$, or say $v$ is a $cross(P)$ vertex. If $v$ belongs to exactly one edge  in $P$, we call $v$ a \emph{free vertex} of $P$, or say $v$ is a $free(P)$ vertex.

\section{Short Path--Proof of Theorem \ref{P2}}
(i). Let $\mathcal{H}$ be a complete $s$-uniform hypergraph on $n$ vertices. It is clear that $ar(n,s,\mathbb{P}_{2})\ge 2$. Suppose that there exists a 2-coloring of $\mathcal{H}$ without a rainbow $\mathbb{P}_{2}$. Then there must be two edges $e_1$ and $e_2$ satisfying that the colors of $e_1$ and $e_2$ are different and $|e_1\cap e_2|>1$. Let $u\in e_1\setminus e_2$ and $v\in e_2\setminus e_1$. Consider the edge $e_3$ consisting of $u$, $v$ and $s-2$ vertices in $V(\mathcal{H})\setminus V(e_1\cup e_2)$. Since there is no rainbow $\mathbb{P}_{2}$, $e_3$ can not be colored with either of the two colors, a contradiction. So any 2-coloring of $\mathcal{H}$ admits a rainbow $\mathbb{P}_{2}$.\\

(ii). Let $\mathcal{H}$ be a complete $s$-uniform hypergraph on $n$ vertices. Consider the following coloring of $\mathcal{H}$ with $\binom{n-2}{s-2}+1$ colors. Take two vertices $u$ and $v$ in $\mathcal{H}$, then the  number of edges containing both $u$ and $v$ is $\binom{n-2}{s-2}$. Coloring each of these edges with different colors and the remaining edges of $\mathcal{H}$ with an additional color. We can see that this coloring of $\mathcal{H}$ yields no rainbow $\mathbb{P}_{3}$. Thus, $ar(n,s,\mathbb{P}_{3})\ge\binom{n-2}{s-2}+2$.

To prove that $ar(n,s,\mathbb{P}_{3})\le\binom{n-2}{s-2}+2$, we argue by contradiction. Suppose that there exists a coloring of $\mathcal{H}$ without a rainbow $\mathbb{P}_{3}$, which uses $\binom{n-2}{s-2}+2$ colors. Let $\mathcal{G}$ be a spanning subgraph of $\mathcal{H}$ with $\binom{n-2}{s-2}+2$ edges such that each color appears on exactly one edge of $\mathcal{G}$. Since $|\mathcal{G}|=\binom{n-2}{s-2}+2>ex(n,s,\mathbb{P}_{2})$ for sufficiently large $n$, there is a linear path $P$ of length two with edges $e_1$ colored by $\alpha_1$ and $e_2$ colored by $\alpha_2$ in $\mathcal{G}$. Let $v$ be the common vertex of $e_1$ and $e_2$. Since  $\mathcal{H}$ contains no  rainbow $\mathbb{P}_{3}$, any edge which contains only one vertex from $(e_1 \cup e_2)\setminus \{v\}$,  must be colored with $\alpha_1$ or $\alpha_2$ in $\mathcal{H}$.

Denote by $\mathcal{F}$ the subgraph obtained by deleting $e_1$ and $e_2$ from $\mathcal{G}$. If there is a linear path $P'$ of length two with edges $f_1$ and $f_2$ in $\mathcal{F}$, let us say the colors of $f_1$ and $f_2$ are $\beta_1$ and $\beta_2$, respectively. If  $f_1$ or $f_2$ contains a $free (P)$ vertex $w$ of $e_1 \cup e_2$ and $w$ is not a $cross(P')$ vertex in $f_1 \cup f_2$, then the edge consisting of $w$ and some $s-1$ vertices in $V(\mathcal{H})\setminus V(e_1 \cup e_2\cup f_1 \cup f_2)$, along with $f_1$ and $f_2$ form a  rainbow $\mathbb{P}_3$. Suppose $f_1 \cup f_2$ contains exactly one  $free (P)$ vertex $w$ of $e_1 \cup e_2$ and $w$ is the $cross(P')$ vertex. Take an edge $e_3$ consisting of a $free (P)$ vertex $x\neq w$ of $e_1$, a $free (P')$ vertex $y$ of $f_1$ and $s-2$ vertices of $V(\mathcal{H})\setminus V(e_1 \cup e_2\cup f_1 \cup f_2)$, then the color of $e_3$ is either $\alpha_1$ or $\alpha_2$. Hence the path with edges $e_3$, $f_1$ and $f_2$ is a  rainbow ${\mathbb{P}_3}$. If $f_1 \cup f_2$ contains no  $free (P)$ vertex  of $e_1 \cup e_2$, then the edge $e_4$ formed by a $free (P)$ vertex $x$ of $e_1$, a $free (P')$ vertex $y$ of $f_1$ and $s-2$ vertices of $V(\mathcal{H})\setminus V(e_1 \cup e_2\cup f_1 \cup f_2)$, must be colored with $\alpha_1$ or $\alpha_2$. So the path with edges $e_4$, $f_1$ and $f_2$ is a  rainbow ${\mathbb{P}_3}$, a contradiction. Therefore, we can assume that there is no ${\mathbb{P}_2}$ in $\mathcal{F}$. By Theorem \ref{turan}, $\mathcal{F}$ consists of all the $\binom{n-2}{s-2}$ edges containing two fixed vertices $x$ and $y$. Note that $\{x,y\}\nsubseteq e_1$, and $\{x,y\}\nsubseteq e_2$ since $e_1,e_2\notin \mathcal{F}$.

We divide our discussion into the following cases depending on the relationship  between vertices $x$, $y$ and edges $e_1$, $e_2$.

\vspace{0.1cm}
\noindent{\em Case $1$.} $x$, $y\notin e_1 \cup e_2$.

Consider the edge $e$ consisting of $x$, $y$, a $free (P)$ vertex  of $e_1 \cup e_2$, and $s-3$ vertices in $V(\mathcal{H})\setminus V(e_1 \cup e_2\cup \{x,y\})$. Then, by the structure of ${\mathcal F}$, we have $e\in \mathcal{F}$, and thus $e$ has a different color with $\alpha_1$ and $\alpha_2$. Therefore, we find a rainbow ${\mathbb{P} _3}$ with edges $e, e_1$ and $e_2$ in $\mathcal{H}$.

\vspace{0.1cm}
\noindent{\em Case $2$.} $x\in  e_1 \cup e_2$, $y\notin  e_1 \cup e_2$, and $x$ is not the $cross (P)$ vertex  in $e_1 \cup e_2$.

The edge $e$, which consists of $x$, $y$ and $s-2$ vertices in $V(\mathcal{H})\setminus V(e_1 \cup e_2\cup \{x,y\})$, has a different color with $\alpha_1$ and $\alpha_2$. Hence, $e$, $e_1$, $e_2$ form a rainbow ${\mathbb{P}_3}$ in $\mathcal{H}$. Note that if $y\in  e_1 \cup e_2$, $x\notin  e_1 \cup e_2$, and $y$ is not the $cross (P)$ vertex  in $e_1 \cup e_2$, we can also find a rainbow ${\mathbb{P}_3}$ in $\mathcal{H}$ similarly.

\vspace{0.1cm}
\noindent{\em Case $3$.} $x$ is a $cross (P)$ vertex  in $e_1 \cup e_2$, $y\notin  e_1 \cup e_2$.

Let $e$ be an edge
with a $free (P)$ vertex  in $e_1$ and $s-1$ vertices in $V(\mathcal{H})\setminus V(e_1 \cup e_2\cup \{x,y\})$. If $e$ has color $\alpha_2$, then we have a rainbow ${\mathbb{P}_2}$  with edges $e$ and $e_1$. Similar to Case 2, we can find a rainbow ${\mathbb{P}_3}$ in $\mathcal{H}$. So the color of $e$ is $\alpha_1$.
Pick an edge $e'$ consisting of $x$, $y$, a vertex $z$ in $e\setminus e_1$ and $s-3$ vertices in $V(\mathcal{H})\setminus V(e_1 \cup e_2\cup \{x,y,z\})$, then $e$, $e'$, $e_2$ form a rainbow ${\mathbb{P}_3}$ in $\mathcal{H}$. And by symmetry, if $y$ is a $cross (P)$ vertex  in $e_1 \cup e_2$ and $x\notin  e_1 \cup e_2$, we can find a rainbow ${\mathbb{P}_3}$ in $\mathcal{H}$ as well.

\vspace{0.1cm}
\noindent{\em Case $4$.}  $x\in e_1\setminus e_2$, $y\in e_2\setminus e_1$.

Take an edge $e$ consisting of  a $free (P)$ vertex $w\neq x$ in $e_1$, and $s-1$ vertices in $V(\mathcal{H})\setminus V(e_1 \cup e_2)$, the color of $e$ is $\alpha_1$
or $\alpha_2$. If the color of $e$ is $\alpha_1$, then the edge $e'$ consisting of $w$, $x$, $y$ and $s-3$ vertices of $V(\mathcal{H})\setminus V(e_1 \cup e_2\cup e \cup \{x,y\})$, along with $e$ and $e_2$ form a rainbow ${\mathbb{P}_3}$. If the color of $e$ is $\alpha_2$, then the edge $e''$ consisting of $x$, $y$ and $s-2$ vertices of $V(\mathcal{H})\setminus V(e_1 \cup e_2\cup e \cup \{x,y\})$, along with $e_1$ and $e$ form a rainbow ${\mathbb{P}_3}$.

 We have examined all the cases in the above discussion. In conclusion, any coloring of $\mathcal{H}$ with $\binom{n-2}{s-2}+2$ colors admits a  rainbow ${\mathbb{P}_3}$. Hence, we have that $ar(n,s,\mathbb{P}_{3})=\binom{n-2}{s-2}+2$.\\

(iii). Since $ar(n,s,\mathcal{B}_{2})\le ar(n,s,\mathcal{P}_{2})\le ar(n,s,\mathbb{P}_{2})$, we can obtain  that $ar(n,s,\mathcal{B}_{2})= ar(n,s,\mathcal{P}_{2})=2$ for $n\ge 3s-4$.\\

(iv).
Let $\mathcal{H}$ be a complete $s$-uniform hypergraph on $n$ vertices. Consider a 3-coloring of $\mathcal{H}$ such that there is no rainbow $\mathcal{P}_{3}$ in $\mathcal{H}$.  Since $ar(n,s,\mathcal{P}_{2})=2<3$ by (iii), there is a rainbow loose path $P$ of length 2 with edges $e_1$ and $e_2$, colored by, say, $\alpha_1$ and $\alpha_2$. Suppose that the number of $free(P)$ vertices in $e_1$ is $a$, so the number of $free(P)$ vertices in $e_2$ is equal to $a$.  Let
\begin{equation*}
p=
\left\{
  \begin{array}{ll}
  \lfloor s/2\rfloor & \hbox{  if $s-a>\lfloor s/2\rfloor$,} \\[2mm]
  s-a & \hbox{  if $s-a\le\lfloor s/2\rfloor$}.
  \end{array}
\right.
\end{equation*}
 Assume that there is an edge $f$ with color $\alpha_3\notin \{\alpha_1,\alpha_2\}$ such that $f\cap (e_1\cup e_2)=\emptyset$. Consider an edge $e$ consisting of all the $free(P)$ vertices in $e_1$,
 $p$ vertices in $f$, and $s-a-p$ vertices in $V(\mathcal{H})\setminus V(e_1\cup e_2\cup f)$. Note that the color of $e$ is either $\alpha_1$ or $\alpha_2$.
 If $e$ is colored with $\alpha_2$, then $e_1,e,f$ is a rainbow $\mathcal{P}_{3}$. So $e$ can only be colored with $\alpha_1$. Similarly, let
 \begin{equation*}
q=
\left\{
  \begin{array}{ll}
 \lceil s/2\rceil & \hbox{  if $s-a>\lceil s/2\rceil$,} \\[2mm]
  s-a & \hbox{  if $s-a\le\lceil s/2\rceil$}.
  \end{array}
\right.
\end{equation*}
Consider the edge $e'$ consisting of  all the $free(P)$ vertices in $e_2$,
 $q$ vertices in $V(f)\setminus V(e)$, and $s-a-q$ vertices in $V(\mathcal{H})\setminus V(e_1\cup e_2\cup f\cup e)$,
then $e'$ is colored with $\alpha_2$. Thus, $e,f,e'$ is a rainbow $\mathcal{P}_{3}$.

So each of the edges colored with $\alpha_3$ contains vertices in $e_1\cup e_2$. Take an edge $h$ with color $\alpha_3$. Note that $h\cap e_1\neq\emptyset$ and $h\cap e_2\neq\emptyset$.

\vspace{0.1cm}
\noindent Case 1. Either $e_1$ or $e_2$ contains a $free(P)$ vertex not belonging to $h$.

Without loss of generality, suppose that there are $b$ $free(P)$ vertices in $e_1\setminus h$, where $b\ge1$. Take an edge $e$ with $b$ $free(P)$ vertices in $e_1\setminus h$ and $s-b$ vertices in $V(\mathcal{H})\setminus V(e_1\cup e_2\cup h)$. Then $e$ is colored with $\alpha_1$ or $\alpha_2$. If $e$ is colored with $\alpha_2$, then $e$, $e_1$, $h$ form a rainbow $\mathcal{P}_3$. So $e$ is colored with $\alpha_1$. Then we have a rainbow $\mathcal{P}_2$ with edges $h$ and $e_2$, which are colored with $\alpha_3$ and $\alpha_2$, respectively. And we have an edge $e$ colored with $\alpha_1$ and $e\cap (h\cup e_2)=\emptyset$. It is the same situation as we analysed before, we can also find a rainbow $\mathcal{P}_3$ in $\mathcal{H}$.

\vspace{0.1cm}
\noindent Case 2. All the $free(P)$ vertices in $e_1\cup e_2$ belong to $h$.

Recall that there are $a$ $free(P)$ vertices in $e_1$. Take an edge $e$ consisting of $a$ $free(P)$ vertices in $e_1$ and $s-a$ vertices in $V(\mathcal{H})\setminus V(e_1\cup e_2\cup h)$. If $e$ is colored with $\alpha_2$, then we have a loose path $P'$ of length two with edges $e$ and $e_1$, which are colored with $\alpha_2$ and $\alpha_1$, respectively, and $e$ contains at least one $free(P')$ vertex not belonging to $h$. It is just similar to Case 1, in which we can  find a rainbow $\mathcal{P}_3$ in $\mathcal{H}$. Thus $e$ is colored with $\alpha_1$. Similarly, take an edge $e'$ consisting of $a$ $free(P)$ vertices in $e_2$ and $s-a$ vertices $V(\mathcal{H})\setminus V(e_1\cup e_2\cup h\cup e)$, we can obtain that $e'$ is colored with $\alpha_2$. { Now, there is a rainbow $\mathcal{P}_3$ consisting of  edges $e$, $h$ and $e'$.}  

Therefore, $ar(n,s,\mathcal{P}_{3})\le3$, for $n\ge 4s-3$. Since $ar(n,s,\mathcal{P}_{3})\ge3$ trivially holds, we have that $ar(n,s,\mathcal{P}_{3})=3$ for $n\ge 4s-3$.

Since $ar(n,s,\mathcal{B}_{3})\le ar(n,s,\mathcal{P}_{3})$, we  obtain  that $ar(n,s,\mathcal{B}_{3})= ar(n,s,\mathcal{P}_{3})=3$ for $n\ge 4s-3$.\qed

\vspace{3mm}

\section{Linear Path--Proof of Theorem \ref{th1}}
Let $\mathcal{H}$ be a complete $s$-uniform hypergraph on $n$ vertices.
For  the lower bounds, we construct a coloring of $\mathcal{H}$ by using the extreme $s$-graphs in Theorem \ref{turan}.

  \begin{prop}\label{Prop:linearlow}
    (a) For $k=2t$, we have $$\min\{ar(n,s,\mathbb{P}_{2t}), ar(n,s,\mathbb{C}_{2t})\}\ge { \binom{n}{s}-\binom{n-t+1}{s}+2}.$$

    (b) For $k=2t+1$, we have $$\min\{ar(n,s,\mathbb{P}_{2t+1}), ar(n,s,\mathbb{C}_{2t+1})\} \ge \binom{n}{s}-\binom{n-t+1}{s}+ \binom{n-t-1}{s-2}+2.$$
  \end{prop}
\begin{proof}
  (a) If $k=2t$, we pick a vertex set $S$ with $t-1$ vertices. Take all the edges that meet $S$ and color each of these edges with different colors. Then color the remaining edges of $\mathcal{H}$ with one additional color. This gives a coloring of $\mathcal{H}$ with $\binom{n}{s}-\binom{n-t+1}{s}+1$ colors. Since each vertex is contained in at most two edges of a rainbow linear path and a rainbow linear cycle,   it is easy to see that any rainbow linear path or rainbow linear cycle  in $\mathcal{H}$ has length at most $2(t-1)+1<2t$. So we have $ar(n,s,\mathbb{P}_{2t})\ge { \binom{n}{s}-\binom{n-t+1}{s}+2}$ and $ar(n,s,\mathbb{C}_{2t})\ge { \binom{n}{s}-\binom{n-t+1}{s}+2}$.

  (b) If $k=2t+1$, we pick a copy of the extreme $\mathbb{P}_{2t}$-free graph obtained in Theorem \ref{turan}. Then color each edge of this extreme $\mathbb{P}_{2t}$-free graph with a distinct color, and color the remaining edges of $\mathcal{H}$ with one additional color to obtain a coloring of $\mathcal{H}$ with $\binom{n}{s}-\binom{n-t+1}{s}+ \binom{n-t-1}{s-2}+1$ colors. It is routine to check that there is no rainbow $\mathbb{P}_{2t+1}$ and no rainbow $\mathbb{C}_{2t+1}$ in the above coloring, and thus  $ar(n,s,\mathbb{P}_{2t+1}) \ge \binom{n}{s}-\binom{n-t+1}{s}+ \binom{n-t-1}{s-2}+2$ and $ar(n,s,\mathbb{C}_{2t+1}) \ge \binom{n}{s}-\binom{n-t+1}{s}+ \binom{n-t-1}{s-2}+2$.
\end{proof}

%

\vspace{1cm}

For the upper bounds, let
\begin{equation*}
D=
\left\{
  \begin{array}{ll}
  \binom{n}{s}-\binom{n-t+1}{s}+2, & \hbox{  if $k = 2t$,} \\[3mm]
  \binom{n}{s}-\binom{n-t+1}{s}+ \binom{n-t-1}{s-2}+2,  & \hbox{  if $k=2t+1$}.
  \end{array}
\right.
\end{equation*}
We argue by contradiction and suppose that there is  a coloring of $\mathcal{H}$ using $D$ colors  yielding no rainbow $\mathbb{P}_{k}$.
Let $\mathcal{G}$ be a spanning subgraph of $\mathcal{H}$ with $D$ edges such that each color appears on exactly one edge
of $\mathcal{G}$. By Theorem \ref{turan}, we obtain that there is a linear path $P$ of length $k-1$ in $\mathcal{G}$. Denote by $e_1, e_2,\ldots, e_{k-1}$ the edges of $P$, and $\alpha_1, \alpha_2,\ldots, \alpha_{k-1}$ the colors of $e_1, e_2,\ldots, e_{k-1}$, respectively.

Since  $\mathcal{H}$ contains no  rainbow ${\mathbb{P}}_{k}$, we obtain the following fact.
\begin{ob}\label{freePcolorsame}
 Let $v$ be a  $free(P)$ vertex in $e_1 \cup e_{k-1}$. Then for any edge $g$ satisfying $g\cap P=\{v\}$, the edge $g$ must be colored with a color of $\{\alpha_1, \alpha_2,\ldots, \alpha_{k-1}\}$.
\end{ob}

Denote by $\mathcal{F}$ the subgraph obtained by deleting $e_1, e_2,\ldots, e_{k-1}$ from $\mathcal{G}$.
We divide the remaining proof into two cases according to the parity of $k$.

\subsection{Completing the proof when $k=2t$ is even}
In this subsection, we assume that $k=2t\ge 4$ is even.

\begin{claim}\label{noPk-1}
When $k=2t\ge 4$, there is no ${\mathbb{P}}_{k-1}$ in  $\mathcal{F}$.
\end{claim}
\begin{proof}
By contradiction,
 suppose there is a linear path $P'$ of length $k-1$ in  $\mathcal{F}$. Denote the edges of $P'$ by $f_1, f_2,\ldots, f_{k-1}$ with colors $\beta_1, \beta_2,\ldots, \beta_{k-1}$, respectively. Since there is no rainbow ${\mathbb{P}}_{k}$ in $\mathcal{H}$, every edge $g$ with $g\cap V(P')=\{u\}$, where $u$ is a  $free(P')$ vertex in $f_1 \cup f_{k-1}$,  must be colored with a color of $\{\beta_1, \beta_2,\ldots, \beta_{k-1}\}$.
We obtain an $s$-graph $\mathcal{F}_e$ by deleting $f_1, f_2,\ldots, f_{k-1}$ and all the edges containing at least two vertices of $P\cup P'$ from $\mathcal{F}$. Let $c$ denote the number of vertices of $P\cup P'$. Then $c\le 2[(k-1)s-(k-2)]$, and so we have $$|\mathcal{F}_e|\ge|\mathcal{F}|-(k-1)-\sum\limits_{i = 2}^s {\binom{c}{i}\binom{n-c}{s-i}}>ex(n,s,\mathbb{P}_{k-2}),$$
for  sufficiently large $n$.  Thus, we have a linear path $P''$ of length $k-2$ in $\mathcal{F}_e$. Denote by $h_1, h_2,\ldots, h_{k-2}$ the edges of $P''$. Moreover, every edge in $P''$ contains at most one vertex from $P\cup P'$. So it follows from Observation \ref{freePcolorsame} that  $P''$ contains no $free(P)$ vertex of $e_1$, $e_{k-1}$, and no  $free(P')$ vertex of $f_1$, $f_{k-1}$. Take an edge $e$ consisting of a $free(P)$ vertex $x$ of $e_1$, a $free(P'')$ vertex of $h_1\setminus V(P\cup P')$ (since $s\ge3$, such vertex does exist), and $s-2$ vertices in $V(\mathcal{H})\setminus V (P\cup P'\cup P'')$, then $e$ is colored with one color in $\{\alpha_1, \alpha_2,\ldots, \alpha_{k-1}\}$ by Observation \ref{freePcolorsame}. Take another edge $e'$ consisting of a $free(P')$ vertex of $f_1\setminus \{x\}$, a $free(P'')$ vertex of $h_{k-2}\setminus V(P\cup P')$ and $s-2$ vertices in $V(\mathcal{H})\setminus (P\cup P'\cup P''\cup e)$, then Observation \ref{freePcolorsame} indicates $e'$ is colored with one color in $\{\beta_1, \beta_2,\ldots, \beta_{k-1}\}$. Hence the path with edges $e,h_1,h_2,\ldots,h_{k-2},e'$ is a rainbow $\mathbb{P}_{k}$, a contradiction. This proves the claim.
\end{proof}

Note that $|\mathcal{F}|\sim (t-1)\binom{n}{s-1}$. By Claim \ref{noPk-1},  Theorems \ref{stability} and \ref{stability2} are applied to $\mathcal{F}$. So we can find an $(s-1)$-graph  $G^*\subset \partial \mathcal{F}$ with  $|G^*|  \sim \binom{n}{s-1}$ for $k\ge 6$ and with $|G^*|  \ge \frac{1}{2}\binom{n}{s-1}$ for $k= 4$, and there is a vertex set $L=\{v_1,v_2,\ldots,v_{t-1}\}$ such that $L\cap V(G^*)=\emptyset$ and $e\cup \{v\}\in \mathcal{F}$ for any $(s-1)$-edge $e\in G^*$  and any $v\in L$.  Moreover,
$|\mathcal{F }- L|= o(n^{s-1})$. We point out that all the vertices of $L$ are not $free(P)$ vertices in $e_1\cup e_{k-1}$.
Otherwise, let $W$ be the vertex set of $P$. By Lemma \ref{lem1}, we can find an $(s-1)$-edge disjoint with $W$ in $G^*$, and it gives an $s$-edge  in $\mathcal{F }$ containing only a $free(P)$ vertex of $P$. This  edge together with $P$ form a rainbow  $\mathbb{P}_{k}$, a contradiction.

\begin{claim}\label{noPk-1}
When $k=2t\ge 4$, there is no  edge in $\mathcal{F}-L$.
\end{claim}
\begin{proof}
Suppose to the contrary there exists an edge $h\in \mathcal{F}-L$ with color $\lambda$ say. By  Lemma \ref{lem1}, we can find two $(s-1)$-edges $a_0$, $b_0$ in $G^*$, such that $a_0$ and $b_0$ have exactly one common vertex $u$ and are disjoint from $P$ and $h$.  Let $W$ be the vertex set of $P\cup h\cup a_0\cup b_0$. By Lemma \ref{lem1}, we can find $(s-1)$-edges $\{a_i,b_i\}$ disjoint from $W$ for  $i=1,\ldots,t-1$, such that for every $i$, $a_i$ and $b_i$ have exactly one common vertex, and for any $j\neq i$,
$\{a_i,b_i\}$ and $\{a_j,b_j\}$ are vertex disjoint. Then,
\begin{eqnarray*}
  &&f_1=a_0\cup \{v_1\}, f_2=\{v_1\}\cup a_1, f_3=b_1\cup \{v_2\}, f_4=\{v_2\}\cup a_2,f_5=b_2\cup \{v_3\},\\
  && \ldots,f_{k-4}=\{v_{t-2}\}\cup a_{t-2},f_{k-3}=b_{t-2}\cup \{v_{t-1}\}, f_{k-2}=\{v_{t-1}\}\cup a_{t-1}
\end{eqnarray*}
form a ${\mathbb{P}}_{k-2}$ in $\mathcal{F}$, denoted by $P'$.  In the rest of the paper, this kind of path would be  abbreviated as
$$P'=a_0\bigoplus_{i=1}^{t-2}(\{v_i\}\oplus a_i\oplus b_i)\oplus\{v_{t-1}\}\oplus a_{t-1}.$$

Let $\beta_1, \beta_2,\ldots, \beta_{k-2}$ be the colors of $f_1, f_2,\ldots, f_{k-2}$ respectively.  Note that $b_0\cup \{v_1\}$ and $b_{t-1}\cup \{v_{t-1}\}$ are edges of $\mathcal{F}$, so both of them  have colors distinct from any other edges in $\mathcal{F}$. The edges, which consist of  one $free(P')$ vertex in $f_1$, one $free(P)$ vertex in $e_1$ and  $s-2$ vertices disjoint with $P$ and $P'$, must be  colored with  colors from $\{\alpha_1, \alpha_2,\ldots, \alpha_{k-1}\}$ by Observation \ref{freePcolorsame}. Let $f$ be an edge consisting of the $free(P')$ vertex $u$ in $f_1$, a vertex in $h$ and $s-2$ vertices disjoint with $P$, $P'$, $b_0$ and $h$. Then the color of $f$ is  in $\{\lambda,\beta_1, \beta_2,\ldots, \beta_{k-2}\}$, because otherwise $h\cup f\cup P'$ is a rainbow $\mathbb{P}_k$. If the color of $f$ is $\lambda$, then we can extend $f\cup P'$ to a rainbow  ${\mathbb{P}}_{k}$ with an additional edge, which containing one $free(P')$ vertex in $f_{k-2}$ and a  $free(P)$ vertex in $e_{k-1}$.

Assume the color
of $f$ is $\beta_j$ for some $j$.  Let $W$ be the vertex set of $f\cup P'\cup h\cup b_0$. By Lemma \ref{lem1}, we can find $(s-1)$-edges $\{a'_i,b'_i\}$ in $G^*$, which are disjoint from $W$ for  $i=1,\ldots,t-1$. Furthermore,
$$h\oplus f\oplus b_0\bigoplus_{i=1}^{t-2}(\{v_i\}\oplus a_i'\oplus b_i')\oplus\{v_{t-1}\}\oplus a_{t-1}'$$
is a rainbow ${\mathbb{P}}_{k}$, a contradiction. This shows that $\mathcal{F}-L$ contains no edge, i.e., all the edges in $\mathcal{F}$  contain vertices in $L$.
\end{proof}

Notice that
 $$|\mathcal{F}|=D-(k-1)=\binom{n}{s}-\binom{n-t+1}{s}-k+3$$
 and there are $\binom{n}{s}-\binom{n-t+1}{s}$ edges in $\mathcal{H}$ which intersect $L$.
Therefore Claim \ref{noPk-1} implies that $\mathcal{F}$ contains no isolated vertices, and
\begin{eqnarray}\label{k-3edges}
  \mbox{there are only  $k-3$ edges containing vertices in $L$ which are not belonging to $\mathcal{F}$.}
\end{eqnarray}

We will derive the final contradiction from the following claim.
\begin{claim}\label{<2disjointL}
When $k=2t\ge 4$, there exist at most one edge in $P$ which is disjoint with $L$.
\end{claim}
\begin{proof}
Suppose that there are two edges $e_i$ and $e_j$ ($j>i$) in $P$, which are disjoint with $L$. If $j>i+1$, we  find an edge $f$ in $\mathcal{F}$ containing a vertex in $e_i$ and disjoint with $e_j$, and an edge $g$ in $\mathcal{F}$ containing a vertex in $e_j$ and disjoint with $e_i$ and $f$. Let $f\cap L=v_p$, $g\cap L=v_q$. Without loss of generality, we suppose that $v_p=v_1$, $v_q=v_{t-1}$. Let $W$ consist of the vertices in $e_i\cup e_j\cup f\cup g$.  By Lemma \ref{lem1}, we can find $(s-1)$-edges $\{a'_i,b'_i\}$ disjoint from $W$ for  $i=1,\ldots,t-1$. Then,
$$e_i\oplus f\oplus a'_1\oplus b'_1\bigoplus_{i=2}^{t-2}(\{v_i\}\oplus a_i'\oplus b_i')\oplus g\oplus e_j$$
is a rainbow ${\mathbb{P}}_{k}$ in $\mathcal{H}$, a contradiction.

If $j=i+1$, we  find an edge $h$ in $\mathcal{F}$ such that $h$ contains exactly one vertex in $e_j\setminus e_i$, and is disjoint with $e_i$. Without loss of generality, we suppose that  $h\cap L=v_1$. Let $W$ consist of the vertices in $e_i\cup e_j\cup h$.  By Lemma \ref{lem1}, we can find $(s-1)$-edges $\{a''_i,b''_i\}$ disjoint from $W$ for  $i=1,\ldots,t-1$. Then,
$$e_i\oplus e_j \oplus h\oplus a''_1\oplus b''_1\bigoplus_{i=2}^{t-2}(\{v_i\}\oplus a_i''\oplus b_i'')\oplus \{v_{t-1}\}\oplus a''_{t-1}$$
is a  rainbow ${\mathbb{P}}_{k}$ in $\mathcal{H}$, a contradiction.
\end{proof}

Since $P$ has $k-1$ edges, Claim \ref{<2disjointL} shows there are at least $k-2$ edges containing vertices of $L$ in $P$.
As $\mathcal{F}=\mathcal{G}-E(P)$, we conclude that there are at least  $k-2$ edges containing vertices in $L$ which are not belonging to $\mathcal{F}$, contradicting (\ref{k-3edges}). This completes the proof for even $k$.

\subsection{Completing the proof when $k=2t+1$ is odd}
In this subsection, we assume that $k=2t+1$ is odd.

Recall that $\mathcal{F}$  denotes   the subgraph obtained from $\mathcal{G}$ by deleting $e_1, e_2,\ldots, e_{k-1}$. If  there is  no ${\mathbb{P}}_{k-1}$ in $\mathcal{F}$, then we can use Theorem \ref{stability} to characterize the structure of $\mathcal{F}$. However, this may not be the case when $k$ is odd. Fortunately, we can prove that after deleting a few edges, the remaining subgraph of  $\mathcal{F}$ contains no  ${\mathbb{P}}_{k-1}$.

\begin{claim}\label{FonePk-1}
 If there is a linear path $P_1$ of length $k-1$ in  $\mathcal{F}$, then $\mathcal{F}-E(P_1)$ contains no  ${\mathbb{P}}_{k-1}$.
\end{claim}
\begin{proof}
Suppose to the contrary that there is a linear path  $P_2$ of length $k-1$ in  $\mathcal{F}-E(P_1)$. Notice  that the colors used in  $P_1$ and $P_2$ are pairwise distinct by the selection of $\mathcal{F}$. Let $f_1, f_2,\ldots, f_{k-1}$ be the edges of $P_1$ with colors  $\beta_1, \beta_2,\ldots, \beta_{k-1}$, respectively,  Denote by $g_1, g_2,\ldots, g_{k-1}$ the edges of $P_2$ with colors  $\gamma_1, \gamma_2,\ldots, \gamma_{k-1}$, respectively.

Let $c$ denote the number of vertices of $P\cup P_1\cup P_2$. Then we have   $c\le 3[(k-1)s-(k-2)]$. Note that the number of edges which contain at least two vertices in $P\cup P_1\cup P_2$ is at most $\sum^{s}_{i=2}{\binom{c}{i}\binom{n-c}{s-i}}$. Since
\begin{eqnarray*}
&&|\mathcal{F}|-\sum^{s}_{i=2}{\binom{c}{i}\binom{n-c}{s-i}}\\&=& \binom{n}{s}-\binom{n-t+1}{s}+ \binom{n-t-1}{s-2}+2-(k-1)
  -\sum^{s}_{i=2}{\binom{c}{i}\binom{n-c}{s-i}}\\
  &>&ex(n,s,\mathbb{P}_{k-3})
\end{eqnarray*}
for  sufficiently large $n$, there exists a linear path $P_3$ of length $k-3$, such that every edge in $P_3$ has at most one vertex of $P\cup P_1\cup P_2$.
Hence, all the $free(P)$ vertices in $e_1\cup e_{k-1}$, $free(P_1)$ vertices in $f_1\cup f_{k-1}$ and $free(P_2)$ vertices in $g_1\cup g_{k-1}$ are not in $P_3$ by Observation \ref{freePcolorsame}. Denote by  $h_1, h_2,\ldots, h_{k-3}$ the edges of $P_3$. Consider an edge $e$, which consists of  a $free(P)$ vertex $x$ in $e_1$, a $free(P_3)$ vertex  in $h_1\setminus (P_1\cup P_2)$ and $s-2$ vertices disjoint with $P\cup P_1\cup P_2\cup P_3$, it follows from Observation \ref{freePcolorsame} that the color of $e$ is  in $\{\alpha_1, \alpha_2,\ldots, \alpha_{k-1}\}$.
And consider an edge $e'$, which consists of  a $free(P_1)$ vertex $y\neq x$ in $f_1\cup f_{k-1}$  (we can find such a vertex $y$ since $s> 3$), a $free(P_3)$ vertex  in $h_{k-3}\setminus (P_1\cup P_2)$ and $s-2$ vertices disjoint with $P_1\cup P_2\cup P_3\cup P\cup e$, then the color of $e'$ is  from $\{\beta_1, \beta_2,\ldots, \beta_{k-1}\}$ by Observation \ref{freePcolorsame}. Moreover, $e\cup P_3\cup e'$ is a rainbow $\mathbb{P}_{k-1}$. Now consider another edge $e''$, which consists of  a $free(P_2)$ vertex $z\neq x, y$ in $g_1\cup g_{k-1}$, a  vertex  in $e'\setminus (P_1\cup P_3)$ and $s-2$ vertices disjoint with $P_1\cup P_2\cup P_3\cup P\cup e\cup e'$, then $e''$ has a color appeared in $e\cup P_3\cup e'$ by Observation \ref{freePcolorsame}. However, to prevent extending $P_2$ to a rainbow $\mathbb{P}_{k}$, the color of $e''$ should be one of $\{\gamma_1, \gamma_2,\ldots, \gamma_{k-3}\}$, a contradiction.
\end{proof}

So if $\mathcal{F}$ has a $\mathbb{P}_{k-1}$, denote by $\mathcal{F}_0$ the subgraph obtained by deleting all the $k-1$ edges of that $\mathbb{P}_{k-1}$ in $\mathcal{F}$.  Then we have $\mathcal{F}_0$ is $\mathbb{P}_{k-1}$-free by Claim \ref{FonePk-1} and $$|\mathcal{F}_0|=\binom{n}{s}-\binom{n-t+1}{s}+ \binom{n-t-1}{s-2}+2-2(k-1).$$
Since $|\mathcal{F}_0|\sim (t-1)\binom{n}{s-1}$ and by Theorem \ref{stability}, we can find an $(s-1)$-graph  $G^*\subset \partial \mathcal{F}_0$ with  $|G^*|  \sim \binom{n}{s-1}$ and a set $L$ of $t-1$ vertices of $\mathcal{F}_0$ such that $L\cap V(G^*)=\emptyset$ and $e\cup \{v\}\in \mathcal{F}_0$ for any $(s-1)$-edge $e\in G^*$  and any $v\in L$.  Moreover,
$|\mathcal{F }_0- L|= o(n^{s-1})$.

If  $\mathcal{F}$ dose not contain a $\mathbb{P}_{k-1}$, then by Theorem \ref{stability} again, we can find an $(s-1)$-graph  $G^*\subset \partial \mathcal{F}$ with  $|G^*|  \sim \binom{n}{s-1}$ and a set $L$ of $t-1$ vertices of $\mathcal{F}$ such that $L\cap V(G^*)=\emptyset$ and $e\cup \{v\}\in \mathcal{F}$ for any $(s-1)$-edge $e\in G^*$  and any $v\in L$.  Additionally,
$|\mathcal{F }- L|= o(n^{s-1})$. Since the number of edges meeting $L$ is at most $\binom{n}{s}-\binom{n-t+1}{s}$, we have $|\mathcal{F }- L|\ge |\mathcal{F}|-\left[\binom{n}{s}-\binom{n-t+1}{s}\right]
=\binom{n-t-1}{s-2}+2-(k-1)>k-1$. Now we delete any $k-1$ edges of $\mathcal{F }- L$ from $\mathcal{F }$ and still denote the remaining subgraph $\mathcal{F}_0$.

Therefore, in either case,  we can find an $(s-1)$-graph  $G^*\subset \partial \mathcal{F}_0$ with  $|G^*|  \sim \binom{n}{s-1}$ and a set $L$ of $t-1$ vertices of $\mathcal{F}_0$ such that $L\cap V(G^*)=\emptyset$ and $e\cup \{v\}\in \mathcal{F}_0$ for any $(s-1)$-edge $e\in G^*$  and any $v\in L$.  Moreover,
$|\mathcal{F }_0- L|= o(n^{s-1})$. We select a $G^*$ with the maximum number of $(s-1)$-edges.
Let the  vertices in $L$ be $v_1,v_2,\ldots,v_{t-1}$. We point out that all the $v_i$ for $i=1,2,\ldots, t-1$ are not $free(P)$ vertices in $e_1\cup e_{k-1}$. Otherwise, let $W$ be the vertex set of $P$. Then by Lemma \ref{lem1}  we can find an $(s-1)$-edge disjoint with $W$ in $G^*$, and this together with $v_i$ will form an $s$-edge which extends $P$ to a rainbow  $\mathbb{P}_{k}$.

Since the number of edges containing vertices of $L$  is at most
 $\binom{n}{s}-\binom{n-t+1}{s}$ in $\mathcal{F}_0$, we have $|\mathcal{F}_0-L|>\binom{n-t-1}{s-2}+2-2(k-1)$.
We further claim the following.

\begin{claim}\label{F0-L}
 $|\mathcal{F}_0-L|\le \binom{n-t-1}{s-2}+\binom{2s-2}{s}+
 \binom{2s-1}{s-1}n$.
\end{claim}
\begin{proof}
By contradiction, assume that $|\mathcal{F}_0-L|>\binom{n-t-1}{s-2}+\binom{2s-2}{s}+
 \binom{2s-1}{s-1}n$. Note that the number of vertices of $\mathcal{F}_0-L$ is  $n-t+1$, by Theorem \ref{turan}, we can find  a ${\mathbb{P}}_{2}$ in $\mathcal{F}_0-L$,  denoted by $P_1$. Let $h_1$, $h_2$ be the edges of $P_1$ with colors $\gamma_1$, $\gamma_2$, respectively. The number of edges containing at least $s-1$ vertices in $P_1$ is  less than $\binom{2s-1}{s}+\binom{2s-1}{s-1}n$. Since $|\mathcal{F}_0-L|-\left[\binom{2s-2}{s}+
 \binom{2s-1}{s-1}n
 \right]
 >ex(n-t
 +1,s,{\mathbb{P}}_{2})$ for  sufficiently large $n$, there is another linear path $P_2$ of length two in $\mathcal{F}_0-L$ such that each edge of which has at least two vertices not in $P_1$. Let $h_3$, $h_4$ be the edges of $P_2$ with colors $\gamma_3$, $\gamma_4$. So $h_4$ contains a $free(P_2)$ vertex $x\notin P_1$. Furthermore, one of $h_1, h_2$ contains a $free(P_1)$ vertex not belonging to $P_2$. Let us say $h_2$ has a  $free(P_1)$ vertex $y\notin P_2$. Take two $(s-1)$-edges $a_0$, $b_0$ in $G^*$ that are disjoint from $P_1$, $P_2$ and $P$, such that $a_0 \cap b_0=u$.  Let $W$ be the vertex set of $P\cup P_1\cup P_2\cup  a_0\cup b_0$. By Lemma \ref{lem1}, we can find $(s-1)$-edges $\{a_i,b_i\}$ disjoint from $W$ for  $i=1,\ldots,t-1$, and so
$$P'=a_0\bigoplus_{i=1}^{t-2}(\{v_i\}\oplus a_i\oplus b_i)\oplus\{v_{t-1}\}\oplus a_{t-1}.$$
is a ${\mathbb{P}}_{2t-2}={\mathbb{P}}_{k-3}$ in $\mathcal{F}_0$.  Denote by  $f_1, f_2,\ldots, f_{k-3}$ the edges of $P'$ with colors $\beta_1, \beta_2,\ldots, \beta_{k-3}$ the colors of each edges in $P'$, respectively. Note that $f_1=a_0\cup \{v_1\}$.

Consider the edge $g$ consisting of $x$, $y$, $u$ and $s-3$ vertices disjoint with $P$, $P_1$, $P_2$ and $b_0$. Then the color of $g$ is in $\{\gamma_1,\ldots\gamma_4,\beta_1,\ldots,\beta_{k-3}\}$, otherwise we can easily extend $P'$ to a rainbow ${\mathbb{P}}_{k}$ by adding $g$, $h_1$ and $h_2$. If the color of $g$ is in $\{\gamma_1,\gamma_2\}$, then $h_3\cup h_4\cup g\cup P'$ is a rainbow ${\mathbb{P}}_{k}$. If the color of $g$ is in $\{\gamma_3,\gamma_4\}$, then $h_1\cup h_2\cup g\cup P'$ is a rainbow ${\mathbb{P}}_{k}$. So the color of $g$ must be in  $\{\beta_1,\ldots,\beta_{k-3}\}$. Let $W$ be the vertex set of $P\cup P_1\cup P_2\cup P'\cup b_0$. By Lemma \ref{lem1}, we can find $(s-1)$-edges $\{a'_i,b'_i\}$ disjoint from $W$ for  $i=1,\ldots,t-1$,
and
$$ h_1\oplus h_2\oplus g\oplus b_0 \bigoplus_{i=1}^{t-2}(\{v_i\}\oplus a_i'\oplus b_i')\oplus\{v_{t-1}\}\oplus a_{t-1}'$$
is a rainbow ${\mathbb{P}}_{k}$, a contradiction.
\end{proof}

Now Claim \ref{F0-L} provides some further structural properties of  $\mathcal{F}_0$.

\begin{claim}\label{F0structual}
 (a) There is no isolated vertex  in $\mathcal{F}_0$.\\
 (b) There are at most
$\binom{2s-2}{s}+
 \binom{2s-1}{s-1}n+2(k-1)-2$ edges meeting $L$ but not in $\mathcal{F}_0$.\\
 (c) Every vertex in $V(\mathcal{F}_0)\setminus L$ belongs to $G^*$, and is not an isolated vertex in $G^*$.
\end{claim}
\begin{proof}
(a) In fact, if $\mathcal{F}_0$ contains an isolated vertex, then the number of edges meeting $L$ in $\mathcal{F}_0$ is at most $\binom{n-1}{s}-\binom{n-1-(t-1)}{s}$.  Thus, we have $$|\mathcal{F}_0-L|\ge |\mathcal{F}_0|-
\left[\binom{n-1}{s}-\binom{n-1-(t-1)}{s}\right]\ge O(n^{s-1}),$$ a contradiction to Claim \ref{F0-L}.  This indicates that $\mathcal{F}_0$ contains  no  isolated vertices.

(b) By Claim \ref{F0-L},  there are  $$|\mathcal{F}_0|-|\mathcal{F}_0-L|\ge\binom{n}{s}-\binom{n-t+1}
 {s}+2-2(k-1)-\left[\binom{2s-2}{s}+
 \binom{2s-1}{s-1}n
 \right]$$ edges in $\mathcal{F}_0$ containing vertices in $L$. Since there are $\binom{n}{s}-\binom{n-t+1}{s}$ edges containing vertices of $L$ in $\mathcal{H}$, we have that there are at most
$\binom{2s-2}{s}+
 \binom{2s-1}{s-1}n+2(k-1)-2$ edges meeting $L$ but not in $\mathcal{F}_0$.

 (c) If there exists a vertex $v\in V(\mathcal{F}_0)\setminus L$ but $v\notin G^*$ or $v$ is an isolated vertex in $G^*$,  then we have at least $\binom{n-(t-1)-1}{s-2}>\binom{2s-2}{s}+
 \binom{2s-1}{s-1}n+2(k-1)-2$ edges meeting $L$ but not belonging to $\mathcal{F}_0$, which is a contradiction to (b). Hence, every vertex in $V(\mathcal{F}_0)\setminus L$ is contained in some edges of  $G^*$ .
\end{proof}

Now we focus on the edges which are disjoint with $L$, namely, the edges in $(\mathcal{F}\cup P)-L=\mathcal{G}-L$ and more generally, the edges in $\mathcal{H}-L$.  Considering the relationship between edges in $\mathcal{H}-L$,
we make the following claim.

\begin{claim}\label{F-Ledges}
 Assume that there exist three edges $f,g,h$ in $\mathcal{H}-L$ with distinct colors such that one of the following holds:\\
 (i) $f,g,h$ form a ${\mathbb{P}}_{3}$;\\
 (ii) $f,g$ form a  ${\mathbb{P}}_{2}$, and $h$ is disjoint with $f\cup g$;\\
 (iii) $f,g,h$  are disjoint with each.\\
 Then we can find a rainbow ${\mathbb{P}}_{k}$ in $\mathcal{H}$.
\end{claim}
\begin{proof}
Notice that for each $e\in \{f,g,h\}$,  there is a unique edge $e'$ in $\mathcal{G}$ having the same color with $e$. So we denote $f', g', h'$ to be the edges in $\mathcal{G}$ with the same color with $f,g,h$, respectively. If the edge $e$ is in $\mathcal{G}$ for some $e\in \{f,g,h\}$, we have $e'=e$.

(i) Assume that  there are three edges $f,g,h$ in $\mathcal{H}-L$ with distinct colors such that
$f,g,h$ form a ${\mathbb{P}}_{3}$.
 Realize that in $G^{*}$, there exists an $(s-1)$-edge $a_0$ containing a vertex $x$ in $h\setminus g$ and disjoint with $(f'\cup g'\cup h'\cup f\cup g\cup h)\setminus\{x\}$. Let $W$ consist of the vertices in $f'\cup g'\cup h'\cup f\cup g\cup h\cup a_0$.  By Lemma \ref{lem1}, we can find $(s-1)$-edges $\{a_i,b_i\}$ disjoint from $W$ for  $i=1,\ldots,t-1$. Then,
$$ f\oplus g\oplus h\oplus a_0 \bigoplus_{i=1}^{t-2}(\{v_i\}\oplus a_i\oplus b_i)\oplus\{v_{t-1}\}\oplus a_{t-1}$$
is a rainbow ${\mathbb{P}}_{k}$ in $\mathcal{H}$.

(ii) Assume that there are three edges $f,g,h$ in $\mathcal{H}-L$ with distinct colors such that $f,g$ form a ${\mathbb{P}}_{2}$, and $h$ is disjoint with $f\cup g$.
In $G^{*}$, there exists an $(s-1)$-edge $a_0$ containing a vertex $x$ in $g\setminus f$ and disjoint with $(f'\cup g'\cup h'\cup g\cup h\cup f)\setminus\{x\}$. And there exists an $(s-1)$-edge $b_0$ containing a vertex $y$ in $h$ and disjoint with $(f'\cup g'\cup h'\cup f\cup g\cup h\cup a_0)\setminus\{y\}$. Let $W$ consist of the vertices in $f'\cup g'\cup h'\cup f\cup g\cup h\cup a_0\cup b_0$.  By Lemma \ref{lem1}, we can find $(s-1)$-edges $\{a_i,b_i\}$ disjoint from $W$ for  $i=1,\ldots,t-1$. Then,
$$ f\oplus g\oplus a_0 \bigoplus_{i=1}^{t-2}(\{v_i\}\oplus a_i\oplus b_i)\oplus\{v_{t-1}\}\oplus b_0\oplus h$$
is a rainbow ${\mathbb{P}}_{k}$ in $\mathcal{H}$.

(iii) Assume that there are three edges $f,g,h$ in $\mathcal{H}-L$ with distinct colors such that they are disjoint with each other.
In $G^{*}$, there exists an $(s-1)$-edge $a_0$ containing a vertex $x$ in $f$ and disjoint with $(f'\cup g'\cup h'\cup f\cup g\cup h)\setminus\{x\}$. And we can find  $(s-1)$-edges $a'_0$, $b'_0$ in $G^{*}$, such that the $(s-1)$-edge $a'_0$ contains a vertex $y_1$ in $g$ and disjoint with $(f'\cup g'\cup h'\cup f\cup g\cup h\cup a_0)\setminus\{y_1\}$; and the $(s-1)$-edge $b'_0$ contains a vertex $y_2\neq y_1$ in $g$ and disjoint with $(f'\cup g'\cup h'\cup f\cup g\cup h\cup a_0\cup a'_0)\setminus\{y_2\}$. Moreover, there exists an $(s-1)$-edge $b_0$ containing a vertex $z$ in $h$ and disjoint with $(f'\cup g'\cup h'\cup f\cup g\cup h\cup a_0\cup a'_0\cup b'_0)\setminus\{z\}$. Let $W$ consist of the vertices in $f'\cup g'\cup h'\cup f\cup g\cup h\cup a_0\cup b_0\cup a'_0\cup b'_0$.  By Lemma \ref{lem1}, we can find $(s-1)$-edges $\{a_i,b_i\}$ disjoint from $W$ for  $i=1,\ldots,t-1$. Then,
$$f\oplus a_0 \oplus \{v_1\}\oplus a'_0\oplus g\oplus b'_0\bigoplus_{i=2}^{t-2}(\{v_i\}\oplus a_{i-1}\oplus b_{i-1})\oplus \{v_{t-1}\}\oplus b_0\oplus h$$
is a rainbow ${\mathbb{P}}_{k}$  in $\mathcal{H}$. Note that in this case, we require that $k=2t+1\ge 7$.
\end{proof}

As $\mathcal{H}$ is a counterexample which contains no rainbow ${\mathbb{P}}_{k}$, we know that each of the conditions (i)(ii)(iii) in Claim \ref{F-Ledges}  cannot exist. Hence there are at most two edges in $P-L$ by Claim \ref{F-Ledges}. In fact, if there are more than two edges in $P-L$, then one of conditions (i), (ii), (iii) in Claim \ref{F-Ledges} must occur. On the other hand, since $|L|=t-1$, there are at most $2(t-1)=k-3$ edges in $P$ containing vertices in $L$. Therefore, there are exactly two edges in $P-L$,  denoted by $e_i$ and $e_j$. Since the number of edges meeting $L$ in $\mathcal{F}$ is at most $\binom{n}{s}-\binom{n-t+1}{s}-(k-3)$, we have
\begin{equation}\label{F-Lextrame}
  |\mathcal{F }- L|\ge |\mathcal{F}|-\left[\binom{n}{s}
-\binom{n-t+1}{s}-(k-3)\right]
=\binom{n-t-1}{s-2}.
\end{equation}
We shall derive the final contradiction depending on $\mathcal{F }- L$ contains a ${\mathbb{P}}_{2}$ or not.

~

\noindent{\bf Case A.} $\mathcal{F }- L$ contains a ${\mathbb{P}}_{2}$.

 We take such a ${\mathbb{P}}_{2}$ in $\mathcal{F }- L$ and denote its edges by $h_1$ and $h_2$ with colors $\gamma_1$ and $\gamma_2$, respectively. Select an edge $e$ in $\mathcal{H}-L$ such that $e$ is disjoint with $\{e_i,e_j,h_1,h_2\}$.  If the color of $e$ is $\alpha_i$ or $\alpha_j$, then  $h_1,h_2$ form a ${\mathbb{P}}_{2}$ and $e$ is disjoint with them, which satisfies condition (ii) of Claim \ref{F-Ledges}, and so we can find a rainbow ${\mathbb{P}}_{k}$ in $\mathcal{H}$.  Assume instead that the color of $e$ is not in $\{\alpha_i, \alpha_j\}$, then $e,e_i,e_j$ have distinct colors.  Furthermore, either $e,e_i,e_j$ are pairwise disjoint, or $e_i,e_j$ form a ${\mathbb{P}}_{2}$ and $e$ is disjoint with them. This satisfies one of the conditions (ii) and (iii) of Claim \ref{F-Ledges}, which we can find a rainbow ${\mathbb{P}}_{k}$ in $\mathcal{H}$.

 ~

\noindent{\bf Case B.} $\mathcal{F }- L$ does not contain a ${\mathbb{P}}_{2}$.

 By (\ref{F-Lextrame}) and Theorem \ref{turan}, $\mathcal{F }- L$ is the extreme ${\mathbb{P}}_{2}$-free hypergraph on $n-t+1$ vertices. Namely, $\mathcal{F }- L$ consists of all the $\binom{n-t-1}{s-2}$ edges containing two fixed vertices $x$ and $y$. Note that $\{x,y\}\nsubseteq e_i$, and $\{x,y\}\nsubseteq e_j$ since $e_i,e_j\notin \mathcal{F}$. If $e_i,e_j$ are not consecutive in $P$, then we select an edge $h$ in $\mathcal{F }- L$ such that $h$ intersects $e_i\cup e_j$ as small as possible. Since $\mathcal{F }- L$ consists of all the $\binom{n-t-1}{s-2}$ edges containing  $x$ and $y$, we have that $h$ intersects $e_i\cup e_j$ in at most two vertices, namely, some of $x$ and $y$. Then $e_i,e_j,h$ must satisfy one of the conditions (i), (ii) and (iii) in Claim \ref{F-Ledges}, which we can find a rainbow ${\mathbb{P}}_{k}$ in $\mathcal{H}$. Assume instead that $e_i,e_j$ are  consecutive in $P$ in the following.

    If either $x\in e_i\setminus e_j$,  $y\in e_j\setminus e_i$ or $e_i\cap e_j\in\{x,y\}$, then we can select an edge $h$ in $\mathcal{F }- L$ such that $h\cap (\{e_i, e_j\}\setminus\{x,y\})=\emptyset$.
 Take an edge $e$ in $\mathcal{H}$ such that $e$ is disjoint with $\{e_i,e_j,h\}$ and $L$.  If the color of $e$ is $\alpha_i$ or $\alpha_j$, then $e_j,h$ form a rainbow ${\mathbb{P}}_{2}$ and $e$ is disjoint with them, or $e_i,h$ form a rainbow ${\mathbb{P}}_{2}$ and $e$ is disjoint with them. Thus  $e,e_j,h$ or $e,e_i,h$ satisfy condition (ii) of Claim \ref{F-Ledges}, which we can  obtain  a rainbow ${\mathbb{P}}_{k}$ in $\mathcal{H}$.  If the color of $e$ is neither $\alpha_i$ nor $\alpha_j$, then $e,e_i,e_j$ have distinct colors. Moreover,  $e_i,e_j$ form a ${\mathbb{P}}_{2}$ and $e$ is disjoint with them. This shows $e,e_i,e_j$ satisfy  condition (ii)  of Claim \ref{F-Ledges}, which we can find a rainbow ${\mathbb{P}}_{k}$ in $\mathcal{H}$.

 Finally,
assume instead that
either $\{x,y\}\cap \{e_i, e_j\}=\emptyset$,  or $x\in e_i\setminus e_j$,  $y\notin e_j$. Then we select an edge $h$ in $\mathcal{F }- L$ such that $h$ intersects $e_i\cup e_j$ as small as possible. Thus $h$ intersects $e_i\cup e_j$ in at most one vertex, namely $x$, which is a free vertex in $e_i\cup e_j$. This indicates  $e_i, e_j,h$ satisfy one of the conditions (i) and (ii) of Claim \ref{F-Ledges}, and hence we can find a rainbow ${\mathbb{P}}_{k}$ in $\mathcal{H}$.

 Therefore, we have established  the upper bound.

\qed\\
\vspace{2mm}

\section{Loose Path--Proof of Theorem \ref{th2}}

Let $\mathcal{H}$ be a complete $s$-uniform hypergraph on $n$ vertices.
The lower bound in Theorem \ref{th2} follows from a similar construction in Theorem \ref{th1} by applying the extreme $s$-graphs obtained from Theorem \ref{turanloose}.

For the upper bound, if $k=2t$, since $ar(n,s,\mathcal{P}_{k})\le ar(n,s,\mathbb{P}_{k})=\binom{n}{s}-\binom{n-t+1}{s}+2$, we have done.

If $k=2t+1$,
we consider, by contradiction,  a coloring of $\mathcal{H}$ using $\binom{n}{s}-\binom{n-t+1}{s}+3$ colors  yielding no rainbow $\mathcal{P}_{k}$.
Let $\mathcal{G}$ be a spanning subgraph of $\mathcal{H}$ with $\binom{n}{s}-\binom{n-t+1}{s}+3$ edges such that each color appears on exactly one edge
of $\mathcal{G}$. By Theorem \ref{turan}, we obtain that there is a loose path $P$ of length $k-1$ in $\mathcal{G}$. Denote by $e_1, e_2,\ldots, e_{k-1}$ the edges of $P$, and $\alpha_1, \alpha_2,\ldots, \alpha_{k-1}$ the colors of $e_1, e_2,\ldots, e_{k-1}$, respectively.

 Denote by $\mathcal{F}$ the subgraph obtained by deleting $e_1, e_2,\ldots, e_{k-1}$ from $\mathcal{G}$. Similar to the proof of Theorem \ref{th1},  we   show that after deleting few edges from $\mathcal{F}$,  the remaining subgraph contains no  ${\mathbb{P}}_{k-1}$.
Actually, we prove something stronger. Call a loose path $P'$ \emph{bad}, if the number of $free(P')$ vertices in the two end edges of $P'$ is at least three. Since $s\ge 3$, it is easy to get that a linear path is also a bad loose path.
\begin{claim}
There are no edge-disjoint bad loose paths of length $k-1$ in  $\mathcal{F}$.
\end{claim}
\begin{proof}
By contradiction, suppose that there are two edge-disjoint bad loose paths $P_1$ and $P_2$ of length $k-1$ in  $\mathcal{F}$. Denote by $f_1, f_2,\ldots, f_{k-1}$ the edges of $P_1$ with colors  $\beta_1, \beta_2,\ldots, \beta_{k-1}$, respectively.  And denote by $g_1, g_2,\ldots, g_{k-1}$ the edges of $P_2$ with colors $\gamma_1, \gamma_2,\ldots, \gamma_{k-1}$, respectively.
Let $c$ denote the number of vertices of $P\cup P_1\cup P_2$. Then $c\le 3[(k-1)s-(k-2)]$. Note that the number of edges which contain at least two vertices in $P\cup P_1\cup P_2$ is at most $\sum^{s}_{i=2}{\binom{c}{i}\binom{n-c}{s-i}}$. Since
\begin{eqnarray*}
|\mathcal{F}|-\sum^{s}_{i=2}{\binom{c}{i}\binom{n-c}{s-i}}>
ex(n,s,\mathbb{P}_{k-3})
\end{eqnarray*}
for  sufficiently large $n$, there exists a linear path $P_3$ of length $k-3$, such that every edge in $P_3$ has at most one vertex of $P\cup P_1\cup P_2$.
Hence, for the same reason as in Observation \ref{freePcolorsame}, all the $free(P)$ vertices in $e_1\cup e_{k-1}$, $free(P_1)$ vertices in $f_1\cup f_{k-1}$ and $free(P_2)$ vertices in $g_1\cup g_{k-1}$ are not in $P_3$. Denote by  $h_1, h_2,\ldots, h_{k-3}$ the edges of $P_3$. Consider the edge $e$, which consists of  a $free(P)$ vertex $x$ in $e_1$, a $free(P_3)$ vertex  in $h_1\setminus (P_1\cup P_2)$ and $s-2$ vertices disjoint with $P\cup P_1\cup P_2\cup P_3$. Then the color of $e$ must be  from $\{\alpha_1, \alpha_2,\ldots, \alpha_{k-1}\}$.
And consider the edge $e'$, which consists of  a $free(P_1)$ vertex $y\neq x$ in $f_1\cup f_{k-1}$  (we can find such a vertex $y$ since $P_1$ is bad), a $free(P_3)$ vertex  in $h_{k-3}\setminus (P_1\cup P_2)$ and $s-2$ vertices disjoint with $P_1\cup P_2\cup P_3\cup P\cup e$, then the color of $e'$ is  from $\{\beta_1, \beta_2,\ldots, \beta_{k-1}\}$. Moreover, $e\cup P_3\cup e'$ is a rainbow $\mathcal{P}_{k-1}$. Now consider another edge $e''$, which consists of  a $free(P_2)$ vertex $z\neq x, y$ in $g_1\cup g_{k-1}$, a  vertex  in $e'\setminus (P_1\cup P_3)$ and $s-2$ vertices disjoint with $P_1\cup P_2\cup P_3\cup P\cup e\cup e'$, then $e''$ must have a color appeared in the rainbow loose path $e\cup P_3\cup e'$. However, to avoid extending $P_2$ to a rainbow $\mathcal{P}_{k}$, the color of $e''$ should be one of $\{\gamma_1, \gamma_2,\ldots, \gamma_{k-3}\}$, a contradiction. Therefore, $\mathcal{F}$ contains no edge-disjoint bad loose paths of length $k-1$.
\end{proof}

So if $\mathcal{F}$ has a bad $\mathcal{P}_{k-1}$, denote by $\mathcal{F}_0$ the subgraph obtained by deleting all the $k-1$ edges of that $\mathcal{P}_{k-1}$ in $\mathcal{F}$; if  $\mathcal{F}$ dose not contain a bad $\mathcal{P}_{k-1}$, then we delete any $k-1$ edges of it and denote the subgraph remained by $\mathcal{F}_0$. Then in either case,
\begin{eqnarray}\label{bad}
  \mbox{$\mathcal{F}_0$ contains no bad $\mathcal{P}_{k-1}$. }
\end{eqnarray}
Therefore, $\mathcal{F}_0$ is $\mathbb{P}_{k-1}$-free and $|\mathcal{F}_0|=\binom{n}{s}-\binom{n-t+1}{s}+3-2(k-1)$.

Note that $|\mathcal{F}_0|\sim (t-1)\binom{n}{s-1}$, by Theorem \ref{stability}, we can find an $(s-1)$-graph  $G^*\subset \partial \mathcal{F}_0$ with  $|G^*|  \sim \binom{n}{s-1}$ and a set $L$ of $t-1$ vertices of $\mathcal{F}_0$ such that $L\cap V(G^*)=\emptyset$ and $e\cup \{v\}\in \mathcal{F}_0$ for any $(s-1)$-edge $e\in G^*$  and any $v\in L$.  Moreover,
$|\mathcal{F }_0- L|= o(n^{s-1})$. Select a $G^*$ with the maximum number of $(s-1)$-edges.  Let the  vertices in $L$ be $v_1,v_2,\ldots,v_{t-1}$. Note that all the $v_i$ for $i=1,2,\ldots, t-1$ are not $free(P)$ vertices in $e_1\cup e_{k-1}$. Otherwise, let $W$ be the vertex set of $P$, by Lemma \ref{lem1}, we can find an $(s-1)$-edge disjoint with $W$ in $G^*$, and then this together with $v_i$ will form an $s$-edge which extends $P$ to a rainbow  $\mathcal{P}_{k}$.

We divide the edges of $\mathcal{F }_0- L$ into two types. Let $Q$ denote the set of  $free(P)$ vertices in $e_1\cup e_{k-1}$. For an edge $e\in \mathcal{F }_0- L$, we call it of {\em Type I}  if $Q\subseteq e$, and of {\em Type II} otherwise. Now we estimate the number of edges of each  type.
\begin{claim}\label{claim2}
There is no $\mathcal{P}_2$ in $\mathcal{F}_0 - L$ whose edges are all of Type II. Therefore, the number of edges of Type II is at most $\lfloor n/s\rfloor$.
\end{claim}
\begin{proof}
Suppose that there is a $\mathcal{P}_2$ with two edges of Type II in  $\mathcal{F}_0 - L$, whose edges are denoted by $h_1$ and $h_2$ with colors  $\gamma_1$ and $\gamma_2$, respectively.
We can take $(s-1)$-edges $a_0$, $b_0$ in $G^*$, such that $a_0$ and $b_0$ have exactly one common vertex $u$ and are disjoint from $P$, $h_1$ and $h_2$. Let $W$ be the vertex set of $P\cup h_1\cup h_2\cup a_0\cup b_0$. By Lemma \ref{lem1}, we can find $(s-1)$-edges $\{a_i,b_i\}$ disjoint from $W$ for  $i=1,\ldots,t-1$, such that for every $i$, $a_i$ and $b_i$ have exactly one common vertex, and for any $j\neq i$,
$\{a_i,b_i\}$ and $\{a_j,b_j\}$ are vertex disjoint. Then,
$$P'=  a_0 \bigoplus_{i=1}^{t-2}(\{v_i\}\oplus a_i\oplus b_i)\oplus \{v_{t-1}\}\oplus a_{t-1}$$
is a ${\mathcal{P}}_{2t-2}={\mathcal{P}}_{k-3}$. Let the edges of $P'$ be $f_1, f_2,\ldots, f_{k-3}$, and the colors of edges  be $\beta_1, \beta_2,\ldots, \beta_{k-3}$ respectively.

Consider an edge $g$, which consists of  a vertex in $h_2\setminus h_1$, the common vertex $u$ of $a_0$ and $b_0$, and $s-2$ vertices disjoint from $P$, $P'$, $h_1$, $h_2$, $b_0$ and $b_{t-1}$. So to prevent extending $P'$ to a rainbow ${\mathcal{P}}_{k}$, the color of $g$ is in $\{\gamma_1,\gamma_2,\beta_1, \beta_2,\ldots, \beta_{k-3}\}$.
If the color of $g$ is $\gamma_1$, consider the edge $e$, which consists of a  vertex in $Q\setminus h_2$, a vertex  in $a_{t-1}$ and $s-2$ vertices disjoint from $P$, $P'$, $h_1$, $h_2$, $g$, $b_0$ and $b_{t-1}$, then the color of $e$ is from $\{\alpha_1, \alpha_2,\ldots, \alpha_{k-1}\}$, and hence $h_2\cup g\cup P'\cup e$ is a rainbow $\mathcal{P}_k$.

 If the color of $g$ is $\gamma_2$, we pick a vertex $w$ in $h_1$, such that if $|Q|<s$, let $w\in h_1\setminus Q$, and if $|Q|\ge s$, let $w$ be an arbitrary vertex in $h_1$. Consider the edge $e'$ consisting of $ w$, the common  vertex  of $a_{t-1}$ and $b_{t-1}$, and $s-2$ vertices disjoint from $P$, $P'$, $h_1$, $h_2$, $g$, $b_0$ and $b_{t-1}$. Then the color of $e'$ is from $\{\gamma_1,\gamma_2,\beta_1, \beta_2,\ldots, \beta_{k-3}\},$ because otherwise, $g\cup P'\cup e'\cup h_1$ is a rainbow  $\mathcal{P}_{k}$.  If  the color of $e'$ is $\gamma_2$, then $h_1\cup e'\cup P'$ is a rainbow $\mathcal{P}_{k-1}$. We obtain a rainbow $\mathcal{P}_k$ by adding an edge $e''$, which consists of   a  vertex in $Q\setminus h_1$, a $free(P')$ vertex  in $f_{1}\setminus \{u\}$ and $s-2$ vertices disjoint from $P$, $P'$, $h_1$, $h_2$, $g$, $e'$, $b_0$ and $b_{t-1}$.  If the color of $e'$ is $\gamma_1$, consider the edge $e'''$ consisting of a  vertex in $Q\setminus (g\cup e')$, a vertex  in $g \setminus (h_2\cup \{u\})$ and $s-2$ vertices disjoint from $P$, $P'$, $h_1$, $h_2$, $g$, $e'$, $b_0$ and $b_{t-1}$. Then the color of $e'''$ is from $\{\alpha_1, \alpha_2,\ldots, \alpha_{k-1}\}$, and hence $e'''\cup g\cup P'\cup e'$ is a rainbow $\mathcal{P}_k$. If the color of $e'$ is $\beta_j$ for some $j$.  Let $W$ be the vertex set of $h_1\cup h_2\cup e'\cup a_0\cup b_0 \cup a_{t-1}\cup b_{t-1}\cup g$, by Lemma \ref{lem1}, we can find $(s-1)$-edges $\{a'_i,b'_i\}$ in $G^*$, which are disjoint from $W$ for  $i=1,\ldots,t-1$, and
$$g\oplus b_0\bigoplus_{i=1}^{t-2}(\{v_i\}\oplus a'_i\oplus b'_i)\oplus \{v_{t-1}\}\oplus b_{t-1} \oplus e'\oplus h_1$$
is a rainbow ${\mathcal{P}}_{k}$.

So assume instead the color of $g$ is one of $\{\beta_1, \beta_2,\ldots, \beta_{k-3}\}$. Let $W$ be the vertex set of $h_1\cup h_2\cup b_0 \cup b_{t-1}\cup g$. By Lemma \ref{lem1}, we can find $(s-1)$-edges $\{a''_i,b''_i\}$ disjoint from $W$ for  $i=1,\ldots,t-1$, and
$$h_1\oplus h_2\oplus g\oplus b_0\bigoplus_{i=1}^{t-2}(\{v_i\}\oplus a''_i\oplus b''_i)\oplus \{v_{t-1}\}\oplus b_{t-1}$$
is a rainbow ${\mathcal{P}}_{k}$.  Hence, there is no $\mathcal{P}_2$ with two edges of Type II, and we have proved  Claim \ref{claim2}.
\end{proof}

Now, we move to Type I edges. Recall that  $Q$ denotes the set of  $free(P)$ vertices in $e_1\cup e_{k-1}$.   So if $s\le|Q|\le2(s-1)$, the number of   Type I edges is at most $1$; if $2 \le|Q|\le s-1$, then a rough counting shows the number of  Type I edges is at most $\binom{n-(t-1)-|Q|}{s-|Q|}\le \binom{n-t-1}{s-2}$.
We further prove that there is no isolated vertex in $\mathcal{F}_0$. Indeed, if $\mathcal{F}_0$ has an isolated vertex, then combining with Claim \ref{claim2},
 $$|\mathcal{F}_0|\le \binom{n-1}{s}-\binom{n-1-(t-1)}{s}+\binom{n-t-1}{s-2}+
 \left\lfloor\frac{n}{s}\right\rfloor,$$
which is less than $\binom{n}{s}-\binom{n-t+1}{s}+3-2(k-1)$ for  sufficiently large $n$, a contradiction.
For the edges of Type I, we have the following claim.
\begin{claim}\label{claim3}
The number of edges of Type I is at most 1.
\end{claim}

\begin{proof}
If $s\le|Q|\le2(s-1)$, then Claim \ref{claim3} follows, and so we may assume  $2 \le|Q|\le s-1$. Suppose to the contrary there are at least two edges of Type I. Then we can find a ${\mathcal{P}}_{2}$ with two edges of Type I, denoted by $h_1$ and $h_2$. Pick vertices $x\in h_2 \setminus h_1$, and $y\in h_1 \setminus h_2$.  The number of edges, which containing exactly one of $\{x,y\}$, one vertex in $L$, and disjoint with $(h_1\cup h_2) \setminus \{x,y\}$, is at least
$$2(t-1)\binom{n-(t-1)-2s}{s-2}>\binom{n-(t-1)-|Q|}{s-2}+
\lfloor\frac{n}{s}\rfloor+2(k-1)-3,$$
which is at least $\binom{n-(t-1)-|Q|}{s-|Q|}+
\lfloor\frac{n}{s}\rfloor+2(k-1)-3$, and so some of them must belong to $\mathcal{F}_0$. Suppose $e\in \mathcal{F}_0$  is such an edge and $v_j\in e\cap L$. Let $W$ be  the vertex set of $ h_1\cup h_2\cup e$. By Lemma \ref{lem1}, we can find $(s-1)$-edges $\{a_i,b_i\}$ disjoint from $W$ for  $i=1,\ldots,t-1$, and
$$h_1\oplus h_2\oplus e\oplus \{a_1,b_1\}\bigoplus_{i=1}^{j-1}(\{v_i\}\oplus a_{i+1}\oplus b_{i+1})\bigoplus_{i=j+1}^{t-2}(\{v_i\}\oplus a_{i}\oplus b_{i})\oplus \{v_{t-1} \}\oplus  a_{t-1}$$
is a bad ${\mathcal{P}}_{k-1}$ in $\mathcal{F}_0$, a contradiction to (\ref{bad}).
\end{proof}

By Claims \ref{claim2} and  \ref{claim3}, we  get that, in  $\mathcal{F}_0$, there are at most $\lfloor n/s\rfloor+1$ edges disjoint with $L$. So there are at least $|\mathcal{F}_0|-\left(\lfloor n/s\rfloor+1\right)=\binom{n}{s}-\binom{n-t+1}{s}+3-2(k-1)-\lfloor n/s\rfloor-1=\binom{n}{s}-\binom{n-t+1}{s}-2(k-1)-\lfloor n/s\rfloor+2$ edges in $\mathcal{F}_0$ containing vertices in $L$. Since there are $\binom{n}{s}-\binom{n-t+1}{s}$ edges containing vertices of $L$ in $\mathcal{H}$, we have that there are at most
$\lfloor n/s\rfloor+2(k-1)-2$ edges meeting $L$ but not in $\mathcal{F}_0$.
If there exists a vertex $v\in V(\mathcal{F}_0)\setminus L$ but $v$ is an isolated vertex in $G^*$, then we have at least $\binom{n-(t-1)-1}{s-2}>\lfloor n/s\rfloor+2(k-1)-2$ edges meeting $L$ but not belonging to $\mathcal{F}_0$, which is a contradiction. Hence, every vertex in $V(\mathcal{F}_0)\setminus L$ is contained in some edges of $G^*$.

{
In fact, to find a rainbow ${\mathcal{P}}_{k}$ in $\mathcal{H}$, we can make use of the suitable edges in $\mathcal{F}_0-L$ and $P-L$ to extend a ${\mathcal{P}}_{k-3}$. We shall prove the following claim, which is analogous to Claim \ref{F-Ledges}.
\begin{claim}\label{claim4}
(i) For $k\ge 7$, if there are three edges $f,g,h$ in $\mathcal{H}-L$ with distinct colors such that  $f, g, h$ form a ${\mathcal{P}}_{3}$, or $f,g$ form a ${\mathcal{P}}_{2}$ and $h$ is disjoint with $f\cup g$, or $f, g, h$ are disjoint with each other, then we can find a rainbow ${\mathcal{P}}_{k}$ in $\mathcal{H}$.

(ii) For $k=5$, if there exists either a rainbow ${\mathcal{P}}_{3}$ or a ${\mathcal{P}}_{2}$ plus a disjoint edge  with all three edges having distinct colors in $\mathcal{H}-L$, then we can find a rainbow ${\mathcal{P}}_{k}$ in $\mathcal{H}$.
\end{claim}
}

\begin{proof}
Similar to the proof of Claim \ref{F-Ledges}, we denote $f', g', h'$ to be the edges in $\mathcal{G}$ with the same color of $f,g,h$, respectively. If the edge $e$ is in $\mathcal{G}$ for some $e\in \{f,g,h\}$, we have $e'=e$.

(i) Let $k\ge 7$. Suppose that there are three edges $f,g,h$ in $\mathcal{H}-L$ with distinct colors such that  $f,g,h$ form a ${\mathcal{P}}_{3}$.  In $G^{*}$, there exists an $(s-1)$-edge $a_0$ containing a vertex $x$ in $h\setminus g$ and disjoint with $(f'\cup g'\cup h'\cup f\cup g\cup h)\setminus\{x\}$. Let $W$ consist of the vertices in $f'\cup g'\cup h'\cup f\cup g\cup h\cup a_0$.  By Lemma \ref{lem1}, we can find $(s-1)$-edges $\{a_i,b_i\}$ disjoint from $W$ for  $i=1,\ldots,t-1$. Then,
$$ f\oplus g\oplus h\oplus a_0 \bigoplus_{i=1}^{t-2}(\{v_i\}\oplus a_{i}\oplus b_{i})\oplus \{v_{t-1}\}\oplus  a_{t-1}$$
is a rainbow ${\mathcal{P}}_{k}$ in $\mathcal{H}$.

If $f,g$ form a ${\mathcal{P}}_{2}$, $h$ is disjoint with $f\cup g$.
In $G^{*}$, there exists an $(s-1)$-edge $a_0$ containing a vertex $x$ in $g\setminus f$ and disjoint with $(f'\cup g'\cup h'\cup g\cup h\cup f)\setminus\{x\}$. And there exists an $(s-1)$-edge $b_0$ containing a vertex $y$ in $h$ and disjoint with $(f'\cup g'\cup h'\cup f\cup g\cup h\cup a_0)\setminus\{y\}$. Let $W$ consist of the vertices in $f'\cup g'\cup h'\cup f\cup g\cup h\cup a_0\cup b_0$.  By Lemma \ref{lem1}, we can find $(s-1)$-edges $\{a_i,b_i\}$ disjoint from $W$ for  $i=1,\ldots,t-1$. Then,
$$ f\oplus g\oplus a_0 \bigoplus_{i=1}^{t-2}(\{v_i\}\oplus a_{i}\oplus b_{i})\oplus \{v_{t-1}\}\oplus b_{0}\oplus h$$
is a rainbow ${\mathcal{P}}_{k}$ in $\mathcal{H}$.

Assume that the three edges $f,g,h$ are disjoint with each other.
In $G^{*}$, there exists an $(s-1)$-edge $a_0$ containing a vertex $x$ in $f$ and disjoint with $(f'\cup g'\cup h'\cup f\cup g\cup h)\setminus\{x\}$. And we can find  $(s-1)$-edges $a'_0$, $b'_0$ in $G^{*}$, such that $a'_0$ contains a vertex $y_1$ in $g$ and disjoint with $(f'\cup g'\cup h'\cup f\cup g\cup h\cup a_0)\setminus\{y_1\}$, $b'_0$ contains a vertex $y_2\neq y_1$ in $g$ and disjoint with $(f'\cup g'\cup h'\cup f\cup g\cup h\cup a_0\cup a'_0)\setminus\{y_2\}$. Moreover, there exists an $(s-1)$-edge $b_0$ containing a vertex $z$ in $h$ and disjoint with $(f'\cup g'\cup h'\cup f\cup g\cup h\cup a_0\cup a'_0\cup b'_0)\setminus\{z\}$. Let $W$ consist of the vertices in $f'\cup g'\cup h'\cup f\cup g\cup h\cup a_0\cup b_0\cup a'_0\cup b'_0$.  By Lemma \ref{lem1}, we can find $(s-1)$-edges $\{a_i,b_i\}$ disjoint from $W$ for  $i=1,\ldots,t-1$. Then,
$$ f\oplus a_0 \oplus \{v_1\}\oplus a'_0\oplus g\oplus b'_0\bigoplus_{i=2}^{t-2}(\{v_i\}\oplus a_{i-1}\oplus b_{i-1})\oplus \{v_{t-1}\}\oplus b_{0}\oplus h$$
is a rainbow ${\mathcal{P}}_{k}$ in $\mathcal{H}$.

(ii) For $k=5$, the proof is identical to (i), and thus omitted.
\end{proof}

Actually, in $P$, there are at most $2(t-1)$ edges containing vertices of $L$, so we can find at least $2$ edges in $P-L$. However, there are not three edges of $P$  satisfying the condition described in Claim \ref{claim4}, and so we derive that there are exactly two edges in  $P-L$.

\noindent{\bf Case A.} $k\ge 7$.

If $|\mathcal{F}_0-L|>0$, then the two edges in  $P-L$ must be consecutive by  Claim \ref{claim4}. Let $e_i$ and $e_{i+1}$ be such two edges. We take an edge $h\in \mathcal{F}_0-L$, then  select  an edge $g\in \mathcal{H}- L$, such that $g$ is disjoint with  $P$ and $h$. If the color of $g$ is $\alpha_j$ for some $j$, then either the color of $g$ is different with $e_i$ or different with $e_{i+1}$. Suppose the colors of $g$ and $e_i$ are different, then we have three edges  $e_i, g, h$ satisfying  the condition of Claim \ref{claim4}, and so we can find a rainbow ${\mathcal{P}}_{k}$ in $\mathcal{H}$.
If the color of $g$ is different with both $e_i$ and $e_{i+1}$, then the three edges $e_i$, $e_{i+1}$ and  $g$ satisfy condition of Claim \ref{claim4}, which we can  find a rainbow ${\mathcal{P}}_{k}$ in $\mathcal{H}$ similarly.

Assume instead that $|\mathcal{F}_0-L|=0$. Then all the $\binom{n}{s}-\binom{n-t+1}{s}+3-2(k-1)$ edges in $\mathcal{F}_0$ contain vertices in $L$, and $P$ has $k-3$ edges containing vertices in $L$. Since the number of edges containing vertices of $L$ is at most $\binom{n}{s}-\binom{n-t+1}{s}$, in $\mathcal{F}\setminus\mathcal{F}_0$ there are at most $k-2$ edges containing vertices in $L$. Since $|\mathcal{F}\setminus\mathcal{F}_0|=k-1$, there is an edge $f$ in $\mathcal{F}\setminus\mathcal{F}_0$ such that $f\cap L=\emptyset$. Furthermore, the color of $f$ is different with any other edges in $\mathcal{F}$.
Applying the same proof to the case that $|\mathcal{F}_0-L|>0$, by replacing the edge $h\in \mathcal{F}_0-L$  with $f$, we can find a rainbow ${\mathcal{P}}_{k}$ in $\mathcal{H}$ as well.\\

{
\noindent{\bf Case B.} $k=5$.

 As noticed above, there are exactly two edges, say $e_i$ and $e_j$, in  $P-L$.

 If $|\mathcal{F}_0-L|=0$, then all the $\binom{n}{s}-\binom{n-t+1}{s}+3-2(k-1)$ edges in $\mathcal{F}_0$ contain vertices in $L$. Also $P$ has $k-3$ edges containing vertices in $L$.   Since the number of edges containing vertices of $L$ is at most $\binom{n}{s}-\binom{n-t+1}{s}$, in $\mathcal{F}\setminus\mathcal{F}_0$ there are at most $k-2$ edges containing vertices in $L$. Since $|\mathcal{F}\setminus\mathcal{F}_0|=k-1$, there is an edge $h$ in $\mathcal{F}\setminus\mathcal{F}_0$ such that $h\cap L=\emptyset$. Then the color of $h$ is different with any other edges in $\mathcal{F}$, and  different with the colors appeared in $P$. If $|\mathcal{F}_0-L|>0$, then there exists an edge $f_1\in \mathcal{F}_0-L$. We set $f$ to be an edge such that  $f=h$ if $|\mathcal{F}_0-L|=0$ and $f=f_1\in \mathcal{F}_0-L$ if $|\mathcal{F}_0-L|>0$.
  Then by Claim \ref{claim4}, $f$ satisfies that
 \begin{eqnarray}\label{casesAB}
   \mbox{either $f\cap e_i=\emptyset$, $f\cap e_j=\emptyset$~~ or ~~$f\cap e_i\neq\emptyset$, $f\cap e_j\neq\emptyset$.}
 \end{eqnarray}

 For the former case of (\ref{casesAB}), pick an edge $g\in \mathcal{H}-L$, such that  $g\cap e_i\neq\emptyset$, $g\cap f\neq\emptyset$ and $g$ is disjoint with $e_j$. Consider the color of $g$. If the color of $g$ is $\alpha_i$, then the three edges  $e_j$, $g$,  $f$ are applied for Claim \ref{claim4}; if the color of $g$ is $\alpha_j$, then the three edges  $e_i$, $g$,  $f$ are applied for Claim \ref{claim4}; if the color of $g$ is different with both $\alpha_i$ and $\alpha_j$, then the three edges  $e_i$, $e_j$, $g$ are applied for Claim \ref{claim4}. Therefore, we can always find  a rainbow ${\mathcal{P}}_{k}$ in this case.

 For the latter case of (\ref{casesAB}) that $f\cap e_i\neq\emptyset$ and $f\cap e_j\neq\emptyset$, we must have $e_i$ and $e_j$ are consecutive in $P$ by Claim \ref{claim4}. Let $g$ be an edge in $\mathcal{H}-L$ such that $g$ is disjoint with $e_i$, $e_j$ and $f$. If the color of $g$ is $\alpha_i$ or $\alpha_j$, then the three edges  $f$, $e_j$, $g$, or the three edges $f$, $e_i$, $g$ are applied for Claim \ref{claim4}. If the color of $g$ is neither $\alpha_i$ nor $\alpha_j$, then the three edges  $e_i$, $e_j$, $g$ are applied for Claim \ref{claim4}, and so we can still find a rainbow ${\mathcal{P}}_{k}$ in $\mathcal{H}$.
This completes the proof of Theorem \ref{th2}.
}

\qed
\vspace{3mm}

\section{Linear Cycle--Proof of Theorem \ref{thc1}}

Let $\mathcal{H}$ be a complete $s$-uniform hypergraph on $n$ vertices. Denote by $V$ the vertex set of $\mathcal{H}$.
Let
\begin{equation*}
g(n,s,k)=
\left\{
  \begin{array}{ll}
  \binom{n}{s}-\binom{n-t+1}{s}+2, & \hbox{  if $k = 2t$,} \\[3mm]
  \binom{n}{s}-\binom{n-t+1}{s}+ \binom{n-t-1}{s-2}+2,  & \hbox{  if $k=2t+1$}.
  \end{array}
\right.
\end{equation*}
To prove Theorem \ref{thc1}, we show that $g(n,s,k)$ is both the lower and upper bound for $ar(n,s,\mathbb{C}_{k})$.

The lower bound follows from Proposition \ref{Prop:linearlow} by constructing a coloring of $\mathcal{H}$  using the extreme $s$-graphs without a $\mathbb{P}_{k-1}$ in Theorem \ref{turan}.

~

For the upper bound,
we argue by contradiction and suppose that there is  a coloring of $\mathcal{H}$ using $g(n,s,k)$ colors  yielding no rainbow $\mathbb{C}_{k}$. Since $g(n,s,k)=ar(n,s,\mathbb{P}_{k})$ and by Theorem \ref{th1}, there is a rainbow linear path $P$ of length $k$  in $\mathcal{H}$.  Let $\mathcal{G}$ be a spanning subgraph of $\mathcal{H}$ with $P\subset \mathcal{G}$, such that $|\mathcal{G}|=g(n,s,k)$ and each color appears on exactly one edge
of $\mathcal{G}$.  Denote by $e_1, e_2,\ldots, e_{k}$ the edges of $P$, and let $\mathcal{F}=\mathcal{G}- \bigcup\limits_{i = 1}^{k - 1} {{e_i}}$.
Clearly, $\mathcal{F}$ is $\mathbb{C}_{k}$-free. The following claim tells us more information about $\mathcal{F}$ when $k=2t+1$.

\begin{claim}\label{claimc1}
When $k=2t+1$, if there  is a linear path $P_1$ of length $k-1$ in  $\mathcal{F}$, then $\mathcal{F}-E(P_1)$ is ${\mathbb{P}}_{k-1}$-free.
 \end{claim}
\begin{proof}
Assume that there  is a linear path $P_1$ of length $k-1$ in  $\mathcal{F}$. Suppose, by contradiction,
that there is a linear path $P_2$ of length $k-1$ in  $\mathcal{F}-E(P_1)$. Denote the edges of $P_1$ by $f_1, f_2,\ldots, f_{k-1}$, and the edges of $P_2$ by $g_1, g_2,\ldots, g_{k-1}$, respectively.
 We obtain an $s$-graph $\mathcal{F}'$ by deleting edge set $E(P_1)\cup E(P_2)$ and all the edges containing at least two vertices of $\bigcup\limits_{i = i}^{k - 1} {{e_i}}\cup E(P_1)\cup E(P_2)$ from $\mathcal{F}$. Let $c$ denote the number of vertices of $V(P)\cup V(P_1)\cup V(P_2)$. Then $c\le ks-(k-1)+2[(k-1)s-(k-2)]$, and so we have
$$|\mathcal{F}'|\ge|\mathcal{F}|-2(k-1)-\sum\limits_{i = 2}^s {\binom{c}{i}\binom{n-c}{s-i}}>ex(n,s,\mathbb{P}_{k-3})$$
for  sufficiently large $n$. Thus, we have a linear path $P_3$ of length $k-3$ in $\mathcal{F}'$. Denote by $h_1, h_2,\ldots, h_{k-3}$ the edges of $P_3$. Note that there are at most $k-3$ vertices in $V(P_3)\cap (V(P)\cup V(P_1)\cup V(P_2))$. Since $s-1\ge k-3+5$ and every path has two disjoint end edges, we can always choose distinct vertices $v_1,v_2,\ldots,v_6$ such that the following holds: $v_1\in e_1$ and  $v_2\in e_{k-1}$ are $free(P)$ vertices, $v_3\in f_1$ and $v_4\in f_{k-1}$ are $free(P_1)$ vertices,  $v_5\in g_1$ and $v_6\in g_{k-1}$ are $free(P_2)$ vertices, and $v_i\notin P_3$ for each $1\le i\le 6$.

Select $u_1\in h_1\setminus(V(P)\cup V(P_1)\cup V(P_2))$ and $u_2\in h_{k-3}\setminus(V(P)\cup V(P_1)\cup V(P_2))$. Consider the edge $e'$ consisting of $v_1,v_2,u_1$ and $s-3$ vertices disjoint with $P\cup P_1\cup P_2\cup P_3$. Then $e'$ has a color appeared in $\bigcup_{i = 1}^{k - 1} {{e_i}}$; otherwise $\bigcup_{i = 1}^{k - 1} {{e_i}}\cup e'$ is a rainbow ${\mathbb{C}}_{k}$, a contradiction. Similarly, the edge $e''$, which consists of $v_3,v_4$, a vertex $x$ in $e'\setminus \{v_1,v_2,u_1\}$ and $s-3$ vertices disjoint with $P\cup P_1\cup P_2\cup P_3\cup e'$, is colored with a color appeared on $P_1$. In addition, consider the edge  $e'''$, which consists of $v_5,v_6,u_2$, a vertex in $e''\setminus \{v_3,v_4,x\}$ and $s-4$ vertices disjoint with $P\cup P_1\cup P_2\cup P_3\cup e'\cup e''$. We get that $e'''$ is colored with a color appeared on $P_2$. Now it follows that  $P_3\cup e'\cup e''\cup e'''$ forms a rainbow ${\mathbb{C}}_{k}$, a contradiction. This proves the claim.
\end{proof}

So when $k=2t+1$, if $\mathcal{F}$ has a $\mathbb{P}_{k-1}$, then we denote by $\mathcal{F}_0$ the subgraph obtained by deleting all the $k-1$ edges of that $\mathbb{P}_{k-1}$ from $\mathcal{F}$, and so  $\mathcal{F}_0$ is $\mathbb{P}_{k-1}$-free by Claim \ref{claimc1}. If there is no $\mathbb{P}_{k-1}$  in $\mathcal{F}$, we delete any $k-1$ edges of $\mathcal{F}$, and denote the subgraph remained by $\mathcal{F}_0$. When $k=2t$, we  obtain $\mathcal{F}_0$ by deleting any $k-1$ edges from $\mathcal{F}$. In any case, we obtain a subgraph $\mathcal{F}_0$ with $|\mathcal{F}_0|=|\mathcal{F}|-[(k-1)s-(k-2)]\sim (t-1)\binom{n}{s-1}$. Moreover, $\mathcal{F}_0$ is $\mathbb{C}_{k}$-free for $k=2t$, and  $\mathcal{F}_0$ is $\mathbb{P}_{k-1}$-free for $k=2t+1$. Thus we can apply Theorem \ref{stability} to $\mathcal{F}_0$ whenever $k$ is even or odd. By Theorem \ref{stability}, we can find an $(s-1)$-graph  $G^*\subset \partial \mathcal{F}_0$ with the maximum number of edges, such that $|G^*|  \sim \binom{n}{s-1}$ and there is  a set $L$ of $t-1$ vertices of $\mathcal{F}_0$ such that $L\cap V(G^*)=\emptyset$ and $e\cup \{v\}\in \mathcal{F}_0$ for any $(s-1)$-edge $e\in G^*$  and any $v\in L$.  Moreover,
$|\mathcal{F }_0- L|= o(n^{s-1})$. Denote $L=\{v_1,v_2,\ldots,v_{t-1}\}$.

 An $s$-edge $e$ is called a \emph{missing-edge} if $e$ contains vertices of $L$ and $e\notin \mathcal{F }_0$. Let $M$ be the set of all the missing-edges, and let $m=|M|$ denote  the number of missing-edges.

 Since $|\mathcal{F }_0|-|\mathcal{F }_0- L|+m=\binom{n}{s}-\binom{n-t+1}{s}$, we have
\begin{equation}\label{lm}
m=
\left\{
  \begin{array}{ll}
  |\mathcal{F }_0- L|-2+2(k-1), & \hbox{  if $k = 2t$,} \\[3mm]
  |\mathcal{F }_0- L|- \binom{n-t-1}{s-2}-2+2(k-1),  & \hbox{  if $k=2t+1$}.
  \end{array}
\right.
\end{equation}
So it follows from $|\mathcal{F }_0- L|= o(n^{s-1})$ that
\begin{equation}\label{m}
m=o(n^{s-1}).
\end{equation}

We  divide the remaining proof into two parts depending on the value of $m$. We will derive contradictions whenever 
$m\le \binom{n-8s-t+1}{s-2}-1$ or
$m> \binom{n-8s-t+1}{s-2}-1$.

\subsection{The case when $m$ is small: $m\le \binom{n-8s-t+1}{s-2}-1$}

In this subsection, we assume that $m\le \binom{n-8s-t+1}{s-2}-1$. The proof applies similar ideas as the proof of Theorem \ref{th1}, where we manage to find certain rainbow path of large length obtained from Lemma \ref{lem1} and then extend to a rainbow ${\mathbb{C}}_{k}$ by selecting some specific edges in $\mathcal{G}-L$. However, the differences with Theorem \ref{th1} is big enough in many details, leading us to rewrite a complete proof of this case.

We start to prove claims below similar to Claims \ref{<2disjointL}, \ref{F0structual} and \ref{F-Ledges}.
\begin{claim}\label{claimc2}
 Every vertex  $v\in V\setminus L$ belongs to $G^*$. Moreover, for any vertex subset $S$ of $V$ with $|S|\le 8s$ and $v\notin S$, there  is  an $(s-1)$-edge $g\in G^*$ such that $v\in g$ and $g$ is disjoint with $S$.
\end{claim}
\begin{proof}
If there is a vertex $v$ such that $v\in V\setminus L$ but $v\notin G^*$, then we have at least $\binom{n-(t-1)-1}{s-2}$ edges meeting $L$, but not belonging to $\mathcal{F}_0$. This implies the number of missing-edges is at least $\binom{n-(t-1)-1}{s-2}>\binom{n-8s-t+1}{s-2}-1$, a contradiction to our assumption that $m\le \binom{n-8s-t+1}{s-2}-1$.

For the `moreover' part, if in $G^*$ every $(s-1)$-set containing $v$ meets $S$, then there are at least $\binom{n-|S|-t+1}{s-2}>m$ missing-edges,  a contradiction. Therefore, there must exist an $(s-1)$-edge $g$ in $G^*$ containing $v$ and  disjoint with $S$.
\end{proof}

\begin{claim}\label{claimc3}
(a) If $k=2t\ge 8$, then there are no two edges $e$, $f$ in $\mathcal{G}-L$, such that $|e\cap f|=1$ or $e\cap f=\emptyset$.

(b) If $k=2t+1\ge 11$, then there are no three edges $e,f,h$ in $\mathcal{G}-L$ satisfying  one of the following conditions:\\
 (i) $e,f,h$ form a ${\mathbb{P}}_{3}$;\\
 (ii) $e,f$ form a  ${\mathbb{P}}_{2}$, and $h$ is disjoint with $e\cup f$;\\
 (iii) $e,f,h$  are pairwise disjoint.
\end{claim}
\begin{proof}
(a) Let $k=2t$.
Suppose to the contrary that there exist two edges $e$, $f$ in $\mathcal{G}-L$ such that $|e\cap f|=1$. Let $u\in e\setminus f$, $v\in f\setminus e$. By Claim \ref{claimc2}, $u,v\in V(G^*)$ and we can find an $(s-1)$-edge $a_0$ in $G^*$ such that $u\in a_0$ and $a_0$ is disjoint with $e\cup f\setminus\{u\}$. Applying Claim \ref{claimc2} again, there is an $(s-1)$-edge $b_0$ in $G^*$ such that $v\in b_0$ and $b_0$ is disjoint with $(e\cup f\cup a_0)\setminus\{v\}$. Let $W$ be the vertex set of $e\cup f\cup a_0\cup b_0$. By Lemma \ref{lem1}, we can find $(s-1)$-edge pairs $\{a_i,b_i\}$ disjoint from $W$ for  $i=1,\ldots,t-1$, such that for every $i$, $a_i$ and $b_i$ have exactly one common vertex, and for any $j\neq i$,
$\{a_i,b_i\}$ and $\{a_j,b_j\}$ are vertex disjoint. Then
$$P'=a_0\bigoplus_{i=1}^{t-2}(\{v_i\}\oplus a_i\oplus b_i)\oplus\{v_{t-1}\}\oplus b_0$$ is a ${\mathbb{P}}_{k-2}$ in $\mathcal{F}_0$.
Adding $e,f$ to $P'$, we obtain a rainbow ${\mathbb{C}}_{k}$, which is a contradiction.

Assume, by contradiction, that there are two edges $e$, $f$ in $\mathcal{G}-L$ such that $e\cap f=\emptyset$. Select four distinct vertices $x,y,z,w$ such that $x,y\in e$ and $z,w\in f$. By Claim \ref{claimc2}, we can find an $(s-1)$-edge $a$ in $G^*$ such that $x\in a$ and $a$ is disjoint with $(e\cup f)\setminus\{x\}$. Applying Claim \ref{claimc2} repeatedly, we can find  $(s-1)$-edges $b$, $a'$, $b'$ one by one in $G^*$ such that $y\in b$ and $b$ is disjoint with $\left(e\setminus\{y\}\right)\cup f\cup a$; $z\in a'$ and $a'$ is disjoint with $e\cup (f\setminus\{z\})\cup a\cup b$; $w\in b'$ and $b'$ is disjoint with $e\cup (f\setminus\{w\})\cup a\cup b\cup a'$. Let $W$ be the vertex set of $e\cup f\cup a\cup b\cup a'\cup b'$. By Lemma \ref{lem1}, we can find $(s-1)$-edge pairs $\{a_i,b_i\}$ disjoint from $W$ for  $i=1,\ldots,t-1$. Then
$$P'=e\oplus b\oplus \{v_1\}\oplus a'\oplus f$$ is a $\mathbb{P}_4$ in $\mathcal{G}$, and
$$P''=a\bigoplus_{i=2}^{t-2}(\{v_i\}\oplus a_{i-1}\oplus b_{i-1})\oplus\{v_{t-1}\}\oplus b'$$ is a $\mathbb{P}_{k-4}$ in $\mathcal{F}_0$. Furthermore, $P'\cup P''$ forms a rainbow $\mathbb{C}_k$, a contradiction. This proves (a).

~

(b) Let $k=2t+1$. We shall derive a contradiction by assuming one of conditions (i)(ii)(iii) holds.

(i) Assume that there is a linear path ${P}_1$ with three consecutive edges $e,f,h$ in $\mathcal{G}-L$. Take two $free(P_1)$ vertices $u,v$ such that $u\in e$ and $v\in h$. By Claim \ref{claimc2}, $u,v\in V(G^*)$ and we can find an $(s-1)$-edge $a_0$ in $G^*$ such that $u\in a_0$ and $a_0$ is disjoint with $(e\setminus\{u\})\cup f\cup h$. Also, there is an $(s-1)$-edge $b_0$ in $G^*$ such that $v\in b_0$ and $b_0$ is disjoint with $e\cup f\cup (h\setminus\{v\})\cup a_0$. Let $W$ be the vertex set of $e\cup f \cup h\cup a_0\cup b_0$. By Lemma \ref{lem1}, we can find $(s-1)$-edge pairs $\{a_i,b_i\}$ disjoint from $W$ for  $i=1,\ldots,t-1$. Then
$$a_0\bigoplus_{i=1}^{t-2}(\{v_i\}\oplus a_i\oplus b_i)\oplus\{v_{t-1}\}\oplus b_0$$ is a $\mathbb{P}_{k-3}$ in $\mathcal{F}_0$.
Adding $e,f,h$ to that $\mathbb{P}_{k-3}$, it results a rainbow ${\mathbb{C}}_{k}$, a contradiction.

(ii) Suppose there are three edges $e,f,h$ in $\mathcal{G}-L$, satisfying that $e,f$ form a  ${\mathbb{P}}_{2}$ and $h$ is disjoint with $e\cup f$. Take four distinct vertices $x,y,z,w$ such that $x\in e\setminus f$, $y\in f\setminus e$, and $z,w\in h$. By Claim \ref{claimc2}, we can find an $(s-1)$-edge $a$ in $G^*$ such that $x\in a$ and $a$ is disjoint with $(e\setminus\{x\})\cup f \cup h$. Applying Claim \ref{claimc2} repeatedly, we can find $(s-1)$-edges $b$, $a'$, $b'$ in $G^*$ such that $y\in b$ and $b$ is disjoint with $e\cup (f\setminus\{y\})\cup a\cup h$; $z\in a'$ and $a'$ is disjoint with $e\cup f \cup (h\setminus\{z\})\cup a\cup b$; $w\in b'$ and $b'$ is disjoint with $e\cup f\cup(h\setminus\{w\})\cup a\cup b\cup a'$. Let $W$ be the vertex set of $e\cup f \cup h\cup a\cup b\cup a'\cup b'$. By Lemma \ref{lem1}, we can find $(s-1)$-edge pairs $\{a_i,b_i\}$ disjoint from $W$ for  $i=1,\ldots,t-1$. Then
$$P'=e\oplus f\oplus b\oplus \{v_1\}\oplus a'\oplus h$$ is a $\mathbb{P}_5$ in $\mathcal{G}$, and
$$P''=a\bigoplus_{i=2}^{t-2}(\{v_i\}\oplus a_{i-1}\oplus b_{i-1})\oplus\{v_{t-1}\}\oplus b'$$ is a $\mathbb{P}_{k-5}$ in $\mathcal{F}_0$. Thus $P'\cup P''$ is a rainbow $\mathbb{C}_k$, a contradiction.

(iii) Suppose that there are three pairwise disjoint edges $e,f,h$  in $\mathcal{G}-L$. Take  distinct vertices $x,y,z,w,u,v$ such that $x, y\in e$, $z,w\in f$, and $u,v\in h$ . By applying Claim \ref{claimc2} repeatedly,  we can find $(s-1)$-edges $a$, $b$, $a'$, $b'$, $a''$, $b''$ in $G^*$ such that $x\in a$ and $a$ is disjoint with $(e\setminus\{x\})\cup f\cup h$; $y\in b$ and $b$ is disjoint with $(e\setminus\{y\})\cup f \cup h\cup a$; $z\in a'$ and $a'$ is disjoint with $e\cup (f \setminus\{z\}) \cup h\cup a\cup b$; $w\in b'$ and $b'$ is disjoint with $e\cup (f \setminus\{w\})  \cup h\cup a\cup b\cup a'$; $u\in a''$ and $a''$ is disjoint with $e\cup f \cup( h\setminus\{u\})\cup a\cup b\cup a'\cup b'$; $v\in b''$ and $b''$ is disjoint with $e\cup f \cup (h\setminus\{v\})\cup a\cup b\cup a'\cup b'\cup a''$. Note that the size of vertex set $S$ in applying Claim \ref{claimc2} is at most $8s$ as required. Let $W$ be the vertex set of $e\cup f \cup h\cup a\cup b\cup a'\cup b'\cup a''\cup b''$. By Lemma \ref{lem1}, we can find $(s-1)$-edges $\{a_i,b_i\}$ disjoint from $W$ for  $i=1,\ldots,t-1$. Then
$$P'=e\oplus b\oplus \{v_1\}\oplus a'\oplus f\oplus b'\oplus \{v_2\}\oplus a''\oplus h$$ is a $\mathbb{P}_7$ in $\mathcal{G}$, and
$$P''=a\bigoplus_{i=3}^{t-2}(\{v_i\}\oplus a_{i-1}\oplus b_{i-1})\oplus\{v_{t-1}\}\oplus b''$$ is a $\mathbb{P}_{k-7}$ in $\mathcal{F}_0$. Therefore, $P'\cup P''$ is a rainbow $\mathbb{C}_k$, a contradiction. This completes the proof of Claim \ref{claimc3}.
\end{proof}

With the aid of Theorem \ref{th1}, it is ready to finish the proof of this case now.
Recall that $P$ is a linear path of length $k$ in $\mathcal{H}$. Since there are at most $2(t-1)$ edges meeting $L$ in $P$, we have that there are at least 2 edges in $P-L\subset \mathcal{G}-L$ when $k=2t$, and at least 3 edges in $P-L$ when  $k=2t+1$. Hence we obtain edges satisfying the conditions of Claim \ref{claimc3}, which is a contradiction.

~

\subsection{The case when $m$ is relatively large: $m> \binom{n-8s-t+1}{s-2}-1$}
In this subsection, we assume that $m> \binom{n-8s-t+1}{s-2}-1$ to complete the proof.
In this case, the  edges meeting $L$ in $\mathcal{F}_0$ may not be enough to make similar arguments as Claim \ref{claimc2}, and it seems that we can not   use Lemma \ref{lem1} directly to find a rainbow  $\mathbb{C}_k$ as before. Our new strategy is to search a dense structure playing similar role as Lemma \ref{lem1}, which is motivated by some ideas in \cite{KMV}. The key ingredient is to find some $(s-2)$-sets of $V$ such that each of them can form rainbow edges with every vertex in $L$ and a large number of other vertices. We shall eventually use these substructures to establish certain desired paths or cycles.

~

For any vertex set $Z$ of a hypergraph $H$, the {\it  degree of $Z$ in $H$},  denoted by $d_H(Z)$, is the number of edges  containing the entire set $Z$ in $H$.

In the following, we denote $K=2s+k-3$ for convenience.
\begin{claim}\label{claimc4}
There are $K$ pairwise disjoint $(s-2)$-sets $T_i$ $(i=1,\ldots,K)$ in $V\setminus L$, such that for every $i\in [K]$ and $j\in [t-1]$ we have
$$d_{\mathcal{F}_0}(T_i\cup\{v_j\})\ge n-s+1-\frac{k(s-1)m}{\binom{n-t+1}{s-2}}.$$
\end{claim}
\begin{proof}
For any $i\in [K]$ and $j\in [t-1]$, let $d_{\overline{\mathcal{F}_0}}(T_i\cup\{v_j\})$ denote the number of $s$-sets $e$ such that $T_i\cup\{v_j\}\subset e$, but $e\notin \mathcal{F}_0$. That is the number of missing-edges containing $T_i\cup\{v_j\}$. Thus we have $d_{\mathcal{F}_0}(T_i\cup\{v_j\})+
d_{\overline{\mathcal{F}_0}}(T_i\cup\{v_j\})=n-s+1$.

Consider an $(s-2)$-set $R$ of $V\setminus L$ selected uniformly randomly from all $(s-2)$-sets of $V\setminus L$. Let $X_i=d_{\overline{\mathcal{F}_0}}(R\cup\{v_i\})$. For every $s$-set $e\notin \mathcal{F}_0$, let
\begin{equation*}
X_i(e)=
\left\{
  \begin{array}{ll}
  1 & \hbox{  if $R\cup \{v_i\} \subset e$,} \\[2mm]
  0 & \hbox{  if $R\cup \{v_i\} \nsubseteq e$}.
  \end{array}
\right.
\end{equation*}
Then the expectation of $X_i$ is
$$\mathbb{E}(X_i)=\sum\limits_{e \notin \mathcal{F}_0} { \mathbb{E}\left( {{X_i}\left( e \right)} \right)}=\sum\limits_{e \in M} { \mathbb{E}\left( {{X_i}\left( e \right)} \right)}  \le \frac{m\binom{s-1}{s-2}}{\binom{n-t+1}{s-2}}
=\frac{(s-1)m}{\binom{n-t+1}{s-2}}.$$
By Markov's inequality, we have
$$\mathrm{Pr}\left[X_i>k\frac{(s-1)m}
{\binom{n-t+1}{s-2}}\right]<\frac{1}{k}.$$
Hence
$$\mathrm{Pr}\left[\exists j\ such \ that \ X_j>k\frac{(s-1)m}
{\binom{n-t+1}{s-2}}\right]\le\sum\limits_{i = 1}^{t - 1} \mathrm{Pr}\left[X_i>k\frac{(s-1)m}
{\binom{n-t+1}{s-2}}\right]< \frac{t-1}{k}<\frac{1}{2},$$
which implies that there are at least $\frac{1}{2}\binom{n-t+1}{s-2}$ such $(s-2)$-sets $R$'s that satisfying $d_{\mathcal{F}_0}(R\cup\{v_i\})= n-s+1-d_{\overline{\mathcal{F}_0}}
(R\cup\{v_i\})\ge n-s+1-k\frac{(s-1)m}
{\binom{n-t+1}{s-2}}$ for all $i\in [t-1]$.
Among those $\frac{1}{2}\binom{n-t+1}{s-2}$ $R$'s, we pick pairwise disjoint $(s-2)$-sets greedily  as many as possible. Let $\ell$ be the largest number that we can pick  pairwise disjoint  $R_1,R_2,\ldots,R_{\ell}$. We show that $\ell\ge K$. In fact, if $\ell<K$, then the number of $(s-2)$-sets meeting
$\bigcup\limits_{j = 1}^{\ell} {{R_j}}$ is at most \[\sum\limits_{r = 1}^{s - 2} {\binom{\ell(s-2)}{r}\binom{n-t+1-\ell(s-2)}{s-2-r}}<
\frac{1}{2}\binom{n-t+1}{s-2}. \]
So we can select $R_{\ell+1}$ from the remained $R$'s such that $R_{\ell+1}$ is disjoint with $\bigcup\limits_{j = 1}^{\ell} {{R_j}}$, a contradiction. Hence $\ell\ge K$ and we can find $K$ $(s-2)$-sets described in Claim \ref{claimc4}.
\end{proof}

~

Let $T= \bigcup\limits_{j = 1}^K {{T_j}}$, and let $U$ denote the vertex subset of $V\setminus (L\cup T)$ such that for every $u\in U$,
$$\text{the edge}~T_i\cup \{v_j\}\cup \{u\}~\text{is belonging to}~ \mathcal{F}_0$$ for all $i\in [K]$, $j\in [t-1]$.

\begin{claim}\label{claim5} We have
$|U|\ge n-K(s-2)-(t-1)-\dfrac{K(t-1)(s-1)k m}{\binom{n-t+1}{s-2}}.$
\end{claim}
\begin{proof}
By Claim \ref{claimc4}, for every $v_j$ and $T_i$, the number of vertex $x$ such that the edge $\{\{x\}\cup\{v_j\}\cup T_i\}\notin \mathcal{F}_0$ is bounded by $n-s+1-d_{\mathcal{F}_0}(T_i\cup\{v_j\})\le k\frac{(s-1)m}
{\binom{n-t+1}{s-2}}$. So we have $|U|\ge n-K(s-2)-(t-1)-K(t-1)\frac{(s-1)k m}{\binom{n-t+1}{s-2}}$ as required.
\end{proof}

The following Claim \ref{claim8} plays the role as Claim \ref{claimc3} in the previous subsection. The difference is that, in Claim \ref{claimc3} we assemble $s$-edges with $(s-1)$-sets in $G^*$ and vertices in $L$, but now we use $T_i$ ($i\in [K]$) and some vertices in $L$ and $U$ to form desired rainbow $s$-edges.  Since $\mathcal{F}_0\subseteq \mathcal{G}- (\bigcup_{i=1}^{k-1} e_i)$, the colors appeared in $\mathcal{F}_0$ are distinct with colors appeared in $\bigcup_{i=1}^{k-1} e_i$. Thus, we can also use edges which have  colors appeared in $\bigcup_{i=1}^{k-1} e_i$, along with edges in $\mathcal{F}_0$, to build rainbow paths and cycles.

Let $J$ denote the set of edges in $E(\mathcal{H}-L)\setminus E(P)$ which are received colors appeared in $\bigcup_{i=1}^{k-1} e_i$.
\begin{claim}\label{claim8}
(1) If $k=2t$, then there are no two edges $e,f\in (\mathcal{F}_0-L)\cup J$ such that $e, f$ form a  rainbow $\mathbb{P}_2$ with $|(f\setminus e ) \cap U| \ge 1$ and $ |(e\setminus f )\cap U|\ge 1$. \\
(2) If $k=2t+1$, then
there are no two edges $e,f\in \mathcal{F}_0-L$ such that $e, f$ form a   $\mathbb{P}_2$, with  $|(f\setminus e ) \cap U| \ge 1$ or $ |(e\setminus f )\cap U|\ge 1$.
\end{claim}
\begin{proof}
 (1) For $k=2t$, if there are two edges $e,f\in ( \mathcal{F}_0-L)\cup J$ such that $e, f$ form a  rainbow $\mathbb{P}_2$ with $|(f\setminus e ) \cap U| \ge 1$ and $ |(e\setminus f )\cap U|\ge 1$. Let $x\in (f\setminus e ) \cap U$ and $y\in (e\setminus f ) \cap U$. Since $T$ is disjoint with $U$,  there are  at most $(2s-3)$ such $T_i$'s that contain vertices of $e$ or $f$.  Hence, there are  at least $K-(2s-3)=k>k-2$ such $T_i$'s that are disjoint with $e$ and $f$.   Suppose, without loss of generality, that $T_i$  is disjoint with $e$ and $f$ for each $i\in[k-2]$.  Then by the definitions of $T$ and $U$,  there is a  rainbow $\mathbb{P}_{k-2}$ in $\mathcal{F}_0$ with edges $h_1,h_2,\ldots, h_{k-2}$,  such that
\begin{eqnarray*}
  &&h_1=\{x\}\cup T_1\cup \{v_1\}, h_2=\{v_1\}\cup T_2\cup \{u_1\}, h_3=\{u_1\}\cup T_3\cup\{v_2\},\\ 
  && \ldots,h_{k-3}=\{u_{t-2}\} \cup T_{k-3} \cup\{v_{t-1}\},
   h_{k-2}=\{v_{t-1}\}\cup T_{k-2} \cup \{y\},
\end{eqnarray*}
where $u_1,\ldots, u_{t-2}$ are distinct vertices selected in $U\setminus (e\cup f)$.
  Adding edges $e,f$ to that  $\mathbb{P}_{k-2}$, we obtain a rainbow $\mathbb{C}_{k}$, a contradiction.

(2) For $k=2t+1$, suppose to the contrary that there are edges $e,f\in \mathcal{F}_0-L$ such that $|e\cap f|=1$ and $x\in (f\setminus e ) \cap U$. As there are at most $(2s-2)$ $T_i$'s containing vertices of $e$ or $f$, we obtain that there are  at least $K-(2s-2)=k-1>k-3$ such $T_i$'s that are disjoint with $e$ and $f$.   Without loss of generality, assume that for $i\in[k-3]$, $T_i$   is disjoint with $e$ and $f$.  Then there is a  $\mathbb{P}_{k-3}$ with the first edge containing $x$, and  all the $k-3$ edges are of the form $T_i\cup\{v_j,u_{\ell}\}$ similar as above, where  $i\in [k-3],j\in[t-1]$ and $u_{\ell}\in U$.
 Adding edges $e,f$ to that  $\mathbb{P}_{k-3}$, we obtain a $\mathbb{P}_{k-1}$ in $\mathcal{F}_0$, a contradiction to Claim \ref{claimc1} that $\mathcal{F}_0$ is $\mathbb{P}_{k-1}$-free for $k=2t+1$.
 \end{proof}

Now we treat the edges of $\mathcal{F}_0-L$ in details. For $0\le i\le s$, let $$B(i)=\{e\in \mathcal{F}_0-L:\ |e\cap U|=i\},$$
and let $B(2^+)=B(2)\cup B(3)\cup\ldots\cup B(s)$.

By  Eq.\eqref{m}, we have $m<\epsilon (n-t-s)^{s-1}<\epsilon n^{s-1}$ for every fixed positive constant $\epsilon$ when $n$ is sufficiently large. Let $\delta=8(t+1)\left[2Kk(t-1)\cdot(s-1)!\right]^{s} \epsilon^{s-2}.$ Here we set the constant $\delta$ to satisfy $0.5<\delta<1$ by selecting an appropriate small constant $\epsilon>0$.

As a consequence of Claim \ref{claim5}, we show the following inequality holds:
\begin{equation}\label{EQB012}
  (t+1)\binom{n-|U|}{s-1}<\frac{\delta}{4}m<\frac{m}{4}.
\end{equation}

In fact, it follows from  Claim \ref{claim5} that
\begin{align*}
(t+1)\binom{n-|U|}{s-1}&\le\frac{2(t+1)}{(s-1)!}
\left(K(s-2)+(t-1)+\frac{K(t-1)(s-1)k m}{\binom{n-t+1}{s-2}}\right)^{s-1}\\
&\le\frac{2^s(t+1)}{(s-1)!}\left[\left(K(s-2)+(t-1)\right)^{s-1}
+\left(\frac{K(t-1)(s-1)km}{\binom{n-t+1}{s-2}}\right)^{s-1}\right]\\
&<\frac{\delta}{8}m
+\frac{2^s(t+1)}{(s-1)!}\left(\frac{K(t-1)(s-1)km}{\frac{(n-t-s)^{s-2}}{(s-2)!}}\right)^{s-1}\\
&<\frac{\delta}{8}m
+(t+1)\left[2Kk(t-1)\cdot(s-1)!\right]^{s}\left(\frac{m}{(n-t-s)^{s-1}}\right)^{s-2}m\\
&<\frac{\delta}{8}m+\frac{\delta}{8}m=\frac{\delta}{4}m<\frac{m}{4},
\end{align*}
where the third line of the inequality holds since the constant $$\frac{2^s(t+1)}{(s-1)!}\left(K(s-2)+(t-1)\right)^{s-1} <\frac{1}{16}\binom{n-8s-t+1}{s-2}\le \frac{\delta}{8}m$$
for sufficiently large $n$.

Therefore, Eq.(\ref{EQB012}) holds, and we will use it to bound the size of $B(0), B(1), B(2^+)$ below.

\begin{claim}\label{claim6}
(a) $|B(0)|<\frac{\delta}{4}m<\frac{m}{4}$.  In particular,  if $m=O(n^{s-2})$, then $|B(0)|\le O(1)$.\\
(b) $|B(1)|<\frac{\delta}{2}m<\frac{m}{2}$.  In particular,  if $m=O(n^{s-2})$, then $|B(1)|\le O(n)$.
\end{claim}
\begin{proof}
(a) Since the edges in $B(0)$ can not form a $\mathbb{C}_k$, by Eq.(\ref{EQB012}), we have $$|B(0)|\le ex(n-|U|,s,\mathbb{C}_k)<(t+1)\binom{n-|U|}{s-1}<\frac{\delta}{4}m<\frac{m}{4}.$$

If $m=O(n^{s-2})$, then by Claim \ref{claim5}, there exists a positive real number $C$ such that
$|U|\ge n-C$. Thus we have $$|B(0)|\le \binom{n-|U|}{s}\le \binom{C}{s}=O(1).$$

(b) Let $Q$ be the collection   of $(s-1)$-sets $h\in V\setminus (U\cup L)$ such that there exists $e\in B(1)$ and $h\subset e$.
We divide $Q$ into two sets $Q_1$ and $Q_2$, such that for every $h\in  Q_1$ there is only one vertex $u\in U$ satisfying $h\cup \{u\}\in  B(1)$, and $Q_2=Q\setminus Q_1$. Hence we have $$|B(1)|<|Q_1|+|Q_2|n.$$
Clearly, $|Q_1|\le\binom{n-|U|}{s-1}<\frac{\delta}{4}m<\frac{m}{4}$ by Eq.(\ref{EQB012}).

To bound $|Q_2|$, notice that there are no $h_1,h_2\in Q_2$ such that $|h_1\cap h_2|=1$. Otherwise, we obtain two $s$-edges $e,f\in \mathcal{F}_0-L$ containing $h_1, h_2$, respectively, where $e, f$ form a  rainbow $\mathbb{P}_2$ with $|(f\setminus e ) \cap U| = 1$ and $ |(e\setminus f )\cap U|= 1$. This   contradicts to Claim \ref{claim8}.
Thus we have $$|Q_2|<ex(n-|U|,s-1,\mathbb{P}_2)\le\binom{n-|U|}{s-3}.$$
Therefore, by Eq.(\ref{EQB012}),
\begin{align*}
|B(1)|&<|Q_1|+|Q_2|n<\frac{\delta}{4}m+\binom{n-|U|}{s-3}n\\
&\le\frac{\delta}{4}m+ (t+1)\binom{n-|U|}{s-1}\\
&<\frac{\delta}{4}m+\frac{\delta}{4}m=\frac{\delta}{2}m<\frac{m}{2}.
\end{align*}

In particularly, if $m=O(n^{s-2})$, then by Claim \ref{claim5}, we have $n\ge |U|\ge n-C$ for  a positive constant number $C$, and so  $$|B(1)|\le \binom{n-|U|}{s-1}\binom{|U|}{1}= O(n).$$

\end{proof}

 Note that \begin{align}\label{B}
|B(2^+)|=|B(2)\cup B(3)\cup\ldots\cup B(s)|&=|\mathcal{F}_0-L|-|B(0)|-|B(1)|.
\end{align}
Next, we explore properties of $|B(2^+)|$ and $m$.

\begin{claim}\label{claimb2k2t}
  (a) We have $|B(2^+)|\le ex(n-t+1, s, \mathbb{P}_2)=\binom{n-t-1}{s-2}$.\\
  (b) We must have $k=2t$ and $m=O(n^{s-2})$. In fact, we have $m< \frac{4}{(s-2)!} n^{s-2}$.
\end{claim}
\begin{proof}
  (a) If $|B(2^+)|> ex(n-t+1, s, \mathbb{P}_2)$, then there is a $\mathbb{P}_2$ with two edges  $e,f\in B(2^+)$. Moreover, by definition of $B(2^+)$, we have $|(f\setminus e ) \cap U| \ge 1$ and $ |(e\setminus f )\cap U|\ge 1$, a  contradiction to Claim \ref{claim8}.

  (b) If $k=2t+1$, then it follows from Eq.(\ref{lm}) (\ref{B}) and Claim \ref{claim6} that
  \begin{align*}
|B(2^+)|>|\mathcal{F}_0-L|-\frac{3}{4}m=\frac{m}{4}+ \binom{n-t-1}{s-2}+2-2(k-1)> ex(n-t+1, s, \mathbb{P}_2),
\end{align*}
a contradiction to Claim \ref{claimb2k2t} (a).

So we must have $k=2t$. Similar as inequality above, by Eq.(\ref{lm}) (\ref{B}) and Claims \ref{claim6} and \ref{claimb2k2t} (a) for $k=2t$, we have
\begin{align*}
ex(n-t+1, s, \mathbb{P}_2)\ge |B(2^+)|>|\mathcal{F}_0-L|-\frac{3}{4}m=\frac{m}{4}+2-2(k-1),
\end{align*}
which shows that $m< \frac{4n^{s-2}}{(s-2)!}$ as required.
\end{proof}

Note that, applying Eq.(\ref{lm}) (\ref{B}) again, Claims \ref{claim6} and \ref{claimb2k2t} provide a further estimation of $|B(2^+)|$ as follows:
\begin{align}\label{B2}
\binom{n-t-1}{s-2}\ge |B(2^+)|
&\ge m-O(n)-O(1)+2-2(k-1).
\end{align}

~

Now we are preparing to find certain edges aiming to lead a contradiction to Claim \ref{claim8}.
\begin{claim}\label{claimxyb2}
  (i) For any $j\in\{1,k-1\}$, there are at least two $free(P)$ vertices in $e_j$, which are not belonging to $L$.\\
  (ii) There exist a $free(P)$ vertex $x\in e_1$ and a $free(P)$ vertex $y\in e_{k-1}$, such that $x,y\notin L$ and not all edges in $B(2^+)$ containing both $x$ and $y$.
\end{claim}
\begin{proof}
  (i) In fact, for any $free(P)$ vertex $v\in e_j\cap L$, and any $free(P)$ vertex $u\in e_{j'}$, where $j,j'\in \{1,k-1\}$ and $j\neq j'$, the edge $g$ consisting of $u,v$ and $s-2$ vertices in $V\setminus V(P)$, must be colored with a color appeared in $\bigcup_{i=1}^{k-1} e_i$; otherwise, we have a rainbow $\mathbb{C}_k$. This indicates that $g$ is a missing-edge. Hence,  if Claim \ref{claimxyb2} (i) dose not hold, then we count the number of missing-edges as
$$m\ge (s-2)(s-3)\binom{n-|V(P)|}{s-2},$$
violating Eq.\eqref{B2}. Hence Claim \ref{claimxyb2} (i) holds.

(ii) By Claim \ref{claimxyb2} (i), assume that $x,z$ are $free(P)$ vertices in $e_1\setminus L$, and $y$ is a $free(P)$ vertex in $e_{k-1}\setminus L$. By contradiction, suppose that  all the edges of $B(2^+)$ contain the vertex  pair $\{x,y\}$, and all the edges of $B(2^+)$ contain  the vertex pair $\{z,y\}$ as well. Then all the edges in $B(2^+)$ contain $\{x,y, z\}$, and so $|B(2^+)|\le  \binom{n-t+1-3}{s-3}$, a contradiction to Eq.(\ref{B2}). This proves Claim \ref{claimxyb2} (ii).
\end{proof}

~

Finally, we are ready to complete the proof.

By  Claim \ref{claimxyb2} (ii), there exist a $free(P)$ vertex $x\in e_1$ and a $free(P)$ vertex $y\in e_{k-1}$, such that $x,y\notin L$ and not all edges in $B(2^+)$ containing both $x$ and $y$. Let $e^*\in B(2^+)$ be such an edge that not containing both $x$ and $y$.

Assume that $x,y\notin e^*$. Since $e^*\in B(2^+)$, we select a vertex $u\in e^*\cap U$, and so $u\notin\{x,y\}$. Consider an $s$-edge $f$ consisting of $x,y,u$ and $s-3$ vertices disjoint with $P$, $L$ and  $e^*$, such that $|f\cap U|\ge 2$. Then we have $f\in J$, otherwise $f\cup P$ forms a rainbow $\mathbb{C}_k$. Thus $f$ and $e^*$ form a rainbow $\mathbb{P}_2$ with  $|(f\setminus e^* ) \cap U| \ge 1$ and $ |(e^*\setminus f )\cap U|\ge 1$, which contradicts  to Claim \ref{claim8} (1).

Assume instead that one of $x,y$ belongs to $e^*$. Without loss of generality, suppose that  $\{x\}= e^*\cap \{x,y\}$.  Consider the edge $g$ consisting of $x,y$ and $s-2$ vertices disjoint with $P$, $L$ and  $e^*$, such that $|(g\setminus \{x\})\cap U|\ge 1$. Then we have $g\in J$ with the same reason as above. Moreover,  $g$ and $e^*$ form a rainbow $\mathbb{P}_2$ with $ |(e^*\setminus f )\cap U|\ge 1$, again contradicting to Claim \ref{claim8} (1).

Therefore, we establish the upper bound and complete the proof of Theorem \ref{thc1}.

\qed\\

\section{Loose cycle--Proof of Theorem \ref{thc2}}

Since the proof of Theorem \ref{thc2} is similar to  Theorem \ref{thc1}, in  this section, we omit some details  and pay more attention to the difference between the proofs of  Theorem \ref{thc1} and Theorem \ref{thc2}.

Let $\mathcal{H}$ be a complete $s$-uniform hypergraph on $n$ vertices. Denote by $V$ the vertex set of $\mathcal{H}$.
The lower bound in Theorem \ref{thc2} follows from a similar construction as Theorem \ref{thc1} by applying the extreme $s$-graphs without $\mathcal{P}_{k-1}$ obtained from Theorem  \ref{turanloose}.

For the upper bound, when $k=2t$, since a loose cycle is also a linear cycle, we have $ar(n,s,\mathcal{C}_{k})\le ar(n,s,\mathbb{C}_{k})=\binom{n}{s}-\binom{n-t+1}{s}+2$, and we are done.

For $k=2t+1$, we shall show below that  how to modify the proof of  the upper bound for anti-Ramsey number of  linear cycles to obtain the upper bound for anti-Ramsey number of  loose cycles.
For loose cycles, we again consider, by contradiction,  a coloring of $\mathcal{H}$ using $\binom{n}{s}-\binom{n-t+1}{s}+3$ colors  yielding no rainbow $\mathcal{C}_{k}$.
Since $ar(n,s,\mathcal{P}_{k})=\binom{n}{s}-\binom{n-t+1}{s}+3$ by Theorem \ref{th2}, there is a rainbow loose path $P$ of length $k$  in $\mathcal{H}$. As before, let $\mathcal{G}$ be a subgraph of $\mathcal{H}$ with $|\mathcal{G}|=\binom{n}{s}-\binom{n-t+1}{s}+3$, such that $P\subset \mathcal{G}$ and each color appears on exactly one edge
of $\mathcal{G}$.  Denote by $e_1, e_2,\ldots, e_{k}$ the edges of $P$, and let $\mathcal{F}=\mathcal{G}- \bigcup\limits_{i = i}^{k - 1} {{e_i}}$.

With  similar argument as  the proof of Claim \ref{claimc1}, we have the following claim.
\begin{claim}\label{claim01}
If $\mathcal{F}$ contains a linear path $P_1$ of length $k-1$, then $\mathcal{F}-E(P_1)$ contains no $\mathbb{P}_{k-1}$.
\end{claim}

If there is a linear path ${P}_1$ of length $k-1$ in $\mathcal{F}$, then we let $\mathcal{F}_0=\mathcal{F}-E(P_1)$; if there is no linear path  of length $k-1$ in $\mathcal{F}$, we delete any $k-1$ edges of $\mathcal{F}$, and denote the subgraph remained by $\mathcal{F}_0$. So we have
$$|\mathcal{F}_0|=|\mathcal{F}|-2(k-1)
=\binom{n}{s}-\binom{n-t+1}
{s}+3-2(k-1),$$
and $\mathcal{F}_0$ is $\mathbb{P}_{k-1}$-free.

Note that $|\mathcal{F}_0|\sim (t-1)\binom{n}{s-1}$. By Theorem \ref{stability}, we can find an $(s-1)$-graph  $G^*\subset \partial \mathcal{F}_0$ with  $|G^*|  \sim \binom{n}{s-1}$ and a set $L$ of $t-1$ vertices of $\mathcal{F}_0$ such that $L\cap V(G^*)=\emptyset$ and $e\cup \{v\}\in \mathcal{F}_0$ for any $(s-1)$-edge $e\in G^*$  and any $v\in L$.  Moreover,
$|\mathcal{F }_0- L|= o(n^{s-1})$. Select a $G^*$ with the maximum number of $(s-1)$-edges.  Denote $L=\{v_1,v_2,\ldots,v_{t-1}\}$ as before.

We still call an $s$-edge $e$ a missing-edge if $e$ contains vertices of $L$ and $e\notin \mathcal{F }_0$. Let $M$ be the set of all the missing-edges, and let $m=|M|$. We have $|\mathcal{F}_0|-|\mathcal{F }_0- L|+m=\binom{n}{s}-\binom{n-t+1}{s}$, and so
\begin{equation*}
m=|\mathcal{F}_0- L|-3+2(k-1).
\end{equation*}

If $m\le \binom{n-8s-t+1}{s-2}-1$, then Claim \ref{claimc2} still holds. Instead of Claim  \ref{claimc3}, we have the following similar claim.
\begin{claim}\label{claim03}
When $k=2t+1\ge 11$, there are no three edges $e,f,h$ in $\mathcal{G}-L$ satisfying that one of the following conditions:\\
 (i) $e,f,h$ form a ${\mathcal{P}}_{3}$;\\
 (ii) $e,f$ form a  ${\mathcal{P}}_{2}$, and $h$ is disjoint with $e\cup f$;\\
 (iii) $e,f,h$  are pairwise disjoint.
\end{claim}
The only difference between  Claim  \ref{claimc3} and    Claim  \ref{claim03} is to  construct a $\mathcal{C}_k$ rather than $\mathbb{C}_k$ to obtain a contradiction, with essentially the same argument (see also Claim \ref{claim4} for details).

Then, as $P$ has length $k$, we can derive that $P-L$ must contain edges satisfying one of the conditions in Claim \ref{claim03}, giving the finial contradiction in the case $m\le \binom{n-8s-t+1}{s-2}-1$.

~

 If
$m> \binom{n-8s-t+1}{s-2}-1$,
 with the arguments that are identical to the linear cycles, Claim \ref{claimc4} and Claim \ref{claim5} hold. By replacing $\mathcal{P}_2$ with $\mathbb{P}_2$, we obtain the following  result similar to Claim \ref{claim8} (2).

 \begin{claim}\label{claim8'}
For $k=2t+1$,
there are no two edges $e,f\in \mathcal{F}_0-L$ such that $e, f$ form a   $\mathcal{P}_2$, $|(f\setminus e ) \cap U| \ge 1$ or $ |(e\setminus f )\cap U|\ge 1$.
\end{claim}

 We still let $B(i)=\{e\in \mathcal{F}_0-L:\ |e\cap U|=i\}$ and $B(2^+)=B(2)\cup B(3)\cup\ldots\cup B(s)$. Then the counting arguments in Claim \ref{claim6} still holds that $|B(0)|<\frac{m}{4}$ and $|B(1)|<\frac{m}{2}$.
 Hence we have
\begin{align}\label{fi}
\notag|B(2^+)|&>|\mathcal{F}_0-L|-\frac{m}{4}-\frac{m}{2}\\
\notag&=m+3-2(k-1)-\frac{3m}{4}\\
&=\frac{m}{4}+3-2(k-1).
\end{align}
By Claim \ref{claim8'}, there are no two edges $f_1,f_2\in B(2^+)$ such that $|f_1\cap U|=i$ and  $|f_2\cap U|=j$  with $i\neq j$. Moreover, for fixed $r$, if $|f_1\cap U|=|f_2\cap U|=r$ for two edges $f_1,f_2\in B(2^+)$, then  $f_1\cap U=f_2\cap U$, i.e., all the edges in  $ B(2^+)$ contain exactly the same $r$ vertices of $U$. Then it follows that
\begin{align*}
|B(2^+)| \le \max_{2\le r\le s}\binom{n-t+1-|U|}{s-r}=\binom{n-t+1-|U|}{s-2}.
\end{align*}
By Claim \ref{claim5} and a similar inequality as (\ref{EQB012}), we have
$$|B(2^+)| \le \binom{n-t+1-|U|}{s-2}<
\frac{\delta}{8}m<\frac{m}{4}+3-2(k-2),$$
which contradicts to Eq.\eqref{fi}.
Hence, we obtain the final contradiction, which proves
 Theorem \ref{thc2}.
 \qed

~

\section{Berge Path and Berge Cycle}
We shall present the proofs of Theorem \ref{th3} and Proposition \ref{BC} on Berge paths and Berge cycles in this section.
\subsection{Berge Path--Proof of Theorem \ref{th3}.}
 For the lower bounds, we will prove  that
  $ar(n,s,\mathcal{B}_{k})\ge \frac{2n}{k}{\lfloor k/2\rfloor \choose s} $ if $k>2s+1$, and
$ar(n,s,\mathcal{B}_{k})\ge \frac{n}{s+1}\lfloor \frac{k-2}{2}\rfloor$  if $3<k\le 2s+1$.
For $k>2s+1$, we partition the $n$ vertices into sets of size $\lfloor k/2\rfloor$ (possibly one of those sets has size smaller than $\lfloor k/2\rfloor$). Denote by $S_1, S_2, \ldots, S_{\ell}$ those obtained sets of size $\lfloor k/2\rfloor$. Then for each $k$-set $S_i$, color each edge contained in $S_i$ with a distinct color. The rest edges are colored with one additional color. It is routine to check that there is no rainbow ${\mathcal B}_k$ in the above coloring. So we have $ar(n,s,\mathcal{B}_{k})\ge \frac{2n}{k}{\lfloor k/2\rfloor \choose s}.$

For $3<k\le 2s+1$, we partition the $n$ vertices into sets of size $s+1$. Then we select $\lfloor k/2\rfloor -1$ edges in each $(s+1)$-set and color each of those edges with a different color. The rest edges are colored with one additional color. Similarly, this provides a $\frac{n}{s+1}\lfloor \frac{k-2}{2}\rfloor$-coloring without a rainbow $\mathcal{B}_{k}$. Hence $ar(n,s,\mathcal{B}_{k})\ge \frac{n}{s+1}\lfloor \frac{k-2}{2}\rfloor.$

For the upper bounds, we will show that
if $ k\ge s+2$,  then for sufficiently large $n$,  $ ar(n,s,\mathcal{B}_{k})\le \frac{n}{k-1}\binom{k-1}{s}+1$; and if $k\le s+1$, then
$ar(n,s,\mathcal{B}_{k})\le \frac{(k-2)n}{s+1}$ for sufficiently large $n$.

(I) For $ k\ge s+2$, let  $\mathcal{H}$ be a complete $s$-uniform hypergraph on $n$ vertices. Consider  a coloring of $\mathcal{H}$ using $\frac{n}{k-1}\binom{k-1}{s}+1$ colors and yielding no rainbow $\mathcal{B}_{k}$. Let $\mathcal{G}$ be a subgraph of $\mathcal{H}$ with $\frac{n}{k-1}\binom{k-1}{s}+1$ edges such that each color appears on exactly one edge
of $\mathcal{G}$. So the number of edges of $\mathcal{G}$ is $|\mathcal{G}|=\frac{n}{k-1}\binom{k-1}{s}+1>ex(n,s,\mathcal{B}_{k-1})$. Hence there is a rainbow Berge path $P$ of length $k-1$ in $\mathcal{G}$. Denote by $e_1$, $e_2$, \ldots, $e_{k-1}$ the edges of $P$ with colors $\alpha_1$, $\alpha_2$, \ldots, $\alpha_{k-1}$, respectively. And there are $k$ vertices $w_1$, $w_2$, \ldots, $w_{k}$ in $P$ such that $w_i,w_{i+1}\in e_i$ for $i=1,\ldots,k-1$. Let $\mathcal{F}$ be the hypergraph obtained by removing all the edges of $P$ from $\mathcal{G}$. We have that $|\mathcal{F}|=\frac{n}{k-1}\binom{k-1}{s}+1-(k-1)=
\frac{n}{k-1}\binom{k-1}{s}-k+2$.

If there is a  Berge path $P^*$ of length $k-1$ in $\mathcal{F}$. Denote by $g_1$, $g_2$, \ldots, $g_{k-1}$ the edges of $P^*$. 
And there are $k$ vertices $z_1$, $z_2$, \ldots, $z_{k}$ in $P^*$ such that $z_i,z_{i+1}\in g_i$ for $i=1,\ldots, k-1$. Then either $w_1\neq z_1$ or $w_1\neq z_{k}$. Without loss of generality, suppose that $w_1\neq z_1$. Consider the edge $e$ consisting of $w_1, z_1$ and $s-2$ vertices in $V(\mathcal{F})\setminus (V(P)\cup V(P^*)) $. If $e$ is colored with a color not in $\{\alpha_1, \alpha_2, \ldots, \alpha_{k-1}\}$, then $e\cup P$ is a rainbow  $\mathcal{B}_{k}$. So $e$ is colored with a color belonging to  $\{\alpha_1, \alpha_2, \ldots, \alpha_{k-1}\}$, then $e\cup P^*$ is a rainbow  $\mathcal{B}_{k}$. Therefore, we have showed that
\begin{eqnarray}\label{noBk-1}
  \mbox{$\mathcal{F}$ contains no $\mathcal{B}_{k-1}$.}
\end{eqnarray}

We further claim that the minimum degree $\delta(\mathcal{F})$ of $\mathcal{F}$ satisfying
\begin{equation}\label{B1}
\delta(\mathcal{F})\ge \frac{1}{k-1}\binom{k-1}{s}-k+1.
\end{equation}
Indeed, if there is a vertex $v$ having degree $d_\mathcal{F}(v)< \frac{1}{k-1}\binom{k-1}{s}-k+1$ in $\mathcal{F}$, then the number of edges in $\mathcal{F}-v$ is more than $|\mathcal{F}|-(\frac{1}{k-1}\binom{k-1}{s}-k+1)=
\frac{n-1}{k-1}\binom{k-1}{s}+1\ge ex(n-1,s, \mathcal{B}_{k-1})+1$ for  sufficiently large $n$. So there is a $\mathcal{B}_{k-1}$ in $\mathcal{F}-v$, which contradicts (\ref{noBk-1}). This proves (\ref{B1}).

Since $|\mathcal{F}|>ex(n,s, \mathcal{B}_{k-2})$  for sufficiently large $n$, there is a  Berge path $P'$ of length $k-2$ in $\mathcal{F}$. Denote by $f_1$, $f_2$, \ldots, $f_{k-2}$ the edges of $P'$ with colors $\beta_1$, $\beta_2$, \ldots, $\beta_{k-2}$, respectively.
And there are $k-1$ vertices $u_1$, $u_2$, \ldots, $u_{k-1}$ in $P'$ such that $u_i,u_{i+1}\in f_i$ for $i=1,\ldots, k-2$.  Since $\mathcal{F}$ contains no $\mathcal{B}_{k-1}$ by (\ref{noBk-1}), the neighbors of $u_1$ and $u_{k-1}$ must belong to $\{u_1, u_2, \ldots, u_{k-1}\}$. In fact, we shall further show in the following claim that
the neighbors of each vertex in $\{ u_2, \ldots, u_{k-2}\}$ also belong to $\{u_1, u_2, \ldots, u_{k-1}\}$.
 Before that, we need the definition of Berge cycles to state the following claim. An
$s$-uniform \emph{Berge cycle} of length $\ell$ is a cyclic list of
distinct $s$-sets $a_1,\ldots , a_\ell$ and $\ell$ distinct vertices
$v_1, \ldots, v_\ell$ such that for each $i=1,2,\ldots, \ell$, $a_i$
contains $v_i$ and $v_{i+1}$ (where $v_{\ell+1} = v_1$).
\begin{eqnarray}\nonumber
  \mbox{If there is a Berge cycle of length $k-1$ and containing the vertices}\\\label{claimb} \mbox{$u_1, u_2, \ldots, u_{k-1}$, then $u_1, u_2, \ldots, u_{k-1}$ constitute a component of $\mathcal{F}$.}
\end{eqnarray}
Suppose that there is a Berge cycle $C$ containing the vertices $u_1, u_2, \ldots, u_{k-1}$. If an edge $f$ in the $C$ contains some vertex $x$ other than $u_1, u_2, \ldots, u_{k-1}$, then deleting $f$ from $C$, we have a $\mathcal{B}_{k-2}$, which can be extended to a $\mathcal{B}_{k-1}$ with edge $f$, contradicting to (\ref{noBk-1}). Thus every edge in the cycle must be contained within
the vertices $u_1, u_2, \ldots, u_{k-1}$. Moreover, for each
vertex $u_i$ in  $C$,  the neighbors of $u_i$ must belong to $\{u_1, u_2, \ldots, u_{k-1}\}$.  Suppose to the contrary that $u_i$ has a neighbor $y$
other than $u_1, u_2, \ldots, u_{k-1}$. Then the edge containing both $u_i$ and  $y$ is not an
edge of $C$, as shown in the argument above. Thus, removing an appropriate edge of $C$
so that we get a path of length $k-2$ with $u_i$ as an endpoint, and hence we can extend this
to a $\mathcal{B}_{k-1}$  with $y$  as an endpoint, a contradiction to (\ref{noBk-1}). This proves (\ref{claimb}).

Now we show that one can always find a Berge cycle of length $k-1$ containing the vertices $u_1, u_2, \ldots, u_{k-1}$. If there is an edge in $\mathcal{F}$ containing both
$u_1$ and $u_{k-1}$, then we can obtain a Berge cycle of length $k-1$. If not, recall that by (\ref{B1}) we have $\delta(\mathcal{F})\ge \frac{1}{k-1}\binom{k-1}{s}-k+1>\binom{\frac{k-1-2}{2}}{s-1}.$ That implies there exist edges $f'$ and $f''$ in $\mathcal{F}$, such that for some $i$, $u_1, u_{i+1}\in f'$ and $u_i, u_{k-1}\in f''$. Thus, we have a Berge cycle of length $k-1$ on the vertices
$$u_1,u_{i+1},u_{i+2},u_{i+3},
\ldots,u_{k-1},u_{i},u_{i-1},u_{i-2},\ldots,u_{1}.$$
Hence, we can find a Berge cycle of length $k-1$ containing the vertices $u_1, u_2, \ldots, u_{k-1}$ in $\mathcal{F}$. By (\ref{claimb}), $u_1, u_2, \ldots, u_{k-1}$ constitute a component of $\mathcal{F}$.

Let $\mathcal{R}$ denote the hypergraph obtained by deleting vertices $u_1, u_2, \ldots, u_{k-1}$ from $\mathcal{F}$. Then $|\mathcal{R}|\ge \frac{n}{k-1}\binom{k-1}{s}-
k+2-\binom{k-1}{s}>ex(n-(k-1),s,\mathcal{B}_{k-2})$ for  sufficiently large $n$. Hence there is a  Berge path $P''$ of length $k-2$ in $\mathcal{R}$. Denote by $h_1$, $h_2$, \ldots, $h_{k-2}$ the edges of $P''$ with colors $\gamma_1$, $\gamma_2$, \ldots, $\gamma_{k-2}$, respectively.
And there are $k-1$ vertices $v_1$, $v_2$, \ldots, $v_{k-1}$ in $P''$ such that $v_i,v_{i+1}\in h_i$ for $i=1,\ldots, k-2$. Note that $\{u_1, u_2, \ldots, u_{k-1}\}\cap\{v_1, v_2, \ldots, v_{k-1}\}=\emptyset$. Since we have either $w_1\notin \{u_1,v_1\}$ or $w_1\notin \{u_{k-1},v_{k-1}\}$,  suppose, without loss of generality, that $w_1\notin \{u_1,v_1\}$ holds. Consider the edge $e'$ with $w_1,u_1,v_1$ and $s-3$ vertices in $V(\mathcal{H})\setminus (V(P)\cup V(P')\cup V(P''))$.
If $s>3$, $e'$ can only be colored with a color in $\{\alpha_1,\alpha_2, \ldots, \alpha_{k-1}\}$, then $h_1\cup e'\cup P'$ is a rainbow $\mathcal{B}_{k}$. If $s=3$, then $e'=\{w_1,u_1,v_1\}$. If the color of $e'$ is not belonging to $\{\beta_1,\beta_2, \ldots, \beta_{k-2}\}\cup\{\gamma_1,\gamma_2, \ldots, \gamma_{k-2}\}$, then $h_1\cup e'\cup P'$ is a rainbow $\mathcal{B}_{k}$. If the color of $e'$ is in $\{\beta_1,\beta_2, \ldots, \beta_{k-2}\}$, let $\tilde{P}=e'\cup P''$,  then $\tilde{P}$ is a rainbow $\mathcal{B}_{k-1}$ in $\mathcal{H}$. Consider an edge $e''=\{w_1,u_1,x\}$, where $x\notin V(P)\cup V(P')\cup V(\tilde{P})$. To prevent extending $P$, the color of  $e''$ must be in $\{\alpha_1,\alpha_2, \ldots, \alpha_{k-1}\}$. However,  to prevent extending $\tilde{P}$,   $e''$ must be colored with a color  from $\{\beta_1,\beta_2, \ldots, \beta_{k-2}\}\cup\{\gamma_1,\gamma_2, \ldots, \gamma_{k-2}\}$, a contradiction. By symmetry, if the color of $e'$ is from $\{\gamma_1,\gamma_2, \ldots, \gamma_{k-2}\}$, we can deduce a similar contradiction as well.  In conclusion,  any coloring of $\mathcal{H}$ using $\frac{n}{k-1}\binom{k-1}{s}+1$ colors  yields a rainbow $\mathcal{B}_{k}$.

~

(II) For $k\le s+1$, let  $\mathcal{H}$ be a complete $s$-uniform hypergraph on $n$ vertices. Consider  a coloring of $\mathcal{H}$ using $\frac{n(k-2)}{s+1}+1$ colors and yielding no rainbow $\mathcal{B}_{k}$. Let $\mathcal{G}$ be a subgraph of $\mathcal{H}$ with $\frac{n(k-2)}{s+1}+1$ edges such that each color appears on exactly one edge
of $\mathcal{G}$. So the number of edges of $\mathcal{G}$ is $|\mathcal{G}|=\frac{n(k-2)}{s+1}+1$. Denote by $\mathcal{C}_1$, $\mathcal{C}_2$, \ldots, $\mathcal{C}_t$ the components of $\mathcal{G}$, and $n_1$, $n_2$, \ldots, $n_t$ the number of vertices of each components, respectively. Then there is a component $\mathcal{C}_i$, such that $|\mathcal{C}_i|>\frac{n_i(k-2)}{s+1}\ge ex(n_i,s,\mathcal{B}_{k-1})$. Hence there is a rainbow Berge path $P$ of length $k-1$ in $\mathcal{C}_i$. Denote by $e_1$, $e_2$, \ldots, $e_{k-1}$ the edges of $P$ with colors $\alpha_1$, $\alpha_2$, \ldots, $\alpha_{k-1}$, respectively. And there are $k$ vertices $w_1$, $w_2$, \ldots, $w_{k}$ in $P$ such that $w_i,w_{i+1}\in e_i$ for $i=1,\ldots,k-1$. Let $\mathcal{F}$ be the hypergraph obtained by removing all the edges of $P$ from $\mathcal{G}$. We have that $|\mathcal{F}|=\frac{(k-2)}{s+1}n+1-(k-1)=
\frac{n(k-2)}{s+1}-k+2>\frac{(k-3)}{s+1}n$.

We will make use of the following result given in \cite{GKL}.
\begin{prop}\cite{GKL}\label{prop1}
Fix $\ell$ and $s$ such that $s\ge \ell>2$. Let $\mathcal{H}$ be a connected $s$-uniform hypergraph with $$|\mathcal{H}|>\frac{\ell-1}{s+1}n$$ edges, where $n$ is the number of vertices in $\mathcal{H}$. Then for each edge $e \in \mathcal{H}$, there is a Berge path of length $\ell$ in $\mathcal{H}$
starting with $e$.
\end{prop}

Let the components of $\mathcal{F}$ be $\mathcal{C}^*_1$, $\mathcal{C}^*_2$, \ldots, $\mathcal{C}^*_\mu$, and $n^*_1$, $n^*_2$, \ldots, $n^*_\mu$ the number of vertices of each components, respectively. Then there is a component $\mathcal{C}^*_j$ satisfying that $|\mathcal{C}^*_j|>\frac{k-3}{s+1}n^*_j
\ge ex(n^*_j,s,\mathcal{B}_{k-2})$. Now we focus on finding a Berge path of length $k-2$ containing some new vertices in $\mathcal{C}^*_j$.

If there exists such a $\mathcal{C}^*_j$ satisfying that $|\mathcal{C}^*_j|>\frac{k-3}{s+1}n^*_j$ and $\mathcal{C}^*_j\cap \mathcal{C}_i=\emptyset$, then we can find a Berge path  of length $k-2$ in $\mathcal{C}^*_j$,  and  its vertices are disjoint with $P$.


If every such $\mathcal{C}^*_j$ with $|\mathcal{C}^*_j|>\frac{k-3}{s+1}n^*_j$ satisfying that $\mathcal{C}^*_j\subseteq \mathcal{C}_i$,
then we have that $n^*_j\ge s+1\ge k$ since the number of vertices of a $\mathcal{B}_{k-2}$ is at least $s-1+2=s+1$.  Furthermore, we claim that
\begin{equation}\label{eq1}
n^*_j\ge  2k.
\end{equation}
In fact, if $n^*_j<  2k$, then $|\mathcal{C}^*_j|\le \binom{n^*_j}{s}<\binom{2k}{s}$. Delete the component $\mathcal{C}^*_j$ from $\mathcal{F}$, we have $n-n^*_j$ vertices and  more than $
\frac{k-2}{s+1}n-k+2-\binom{2k}{s}>
\frac{k-3}{s+1}(n-k)\ge\frac{k-3}{s+1}(n-n^*_j)$ edges. So there is a component $\mathcal{C}^*_t$, such that $|\mathcal{C}^*_t|>\frac{k-3}{s+1}n^*_t
\ge ex(n^*_t,s,\mathcal{B}_{k-2})$ edges and $\mathcal{C}^*_t\cap \mathcal{C}_i=\emptyset$, a contradiction. So  we have $n^*_j\ge  2k$, which proves (\ref{eq1}). Hence, there is a vertex $u_1\notin \{w_1, w_2, \ldots, w_{k}\}$ in  $\mathcal{C}^*_j$. We take an edge $e$ in  $\mathcal{C}^*_j$ containing $u_1$, by Proposition \ref{prop1}, there is a $\mathcal{B}_{k-2}$ starting with $e$.

In both cases above, we denote by $P'$ the  $\mathcal{B}_{k-2}$ we obtained, and denote by $f_1$, $f_2$, \ldots, $f_{k-2}$ the edges of $P'$ with colors $\beta_1$, $\beta_2$, \ldots, $\beta_{k-2}$, respectively. There are $k-1$ vertices $u_1$, $u_2$, \ldots, $u_{k-1}$ in $P'$ such that $u_i,u_{i+1}\in f_i$ for $i=1,\ldots, k-2$. Note that $u_1\notin \{w_1, w_2, \ldots, w_{k}\}$.
Let $\mathcal{R}$ denote the hypergraph obtained by deleting $f_1, f_2, \ldots, f_{k-2}$ from $\mathcal{F}$. Then $|\mathcal{R}|=\frac{k-2}{s+1}n-k+2-(k-2)>\frac{k-3}{s+1}n$. Let the components of $\mathcal{R}$ be $\mathcal{C}^{**}_1$, $\mathcal{C}^{**}_2$, \ldots, $\mathcal{C}^{**}_\tau$, and $n^{**}_1$, $n^{**}_2$, \ldots, $n^{**}_\tau$ the number of vertices of each component, respectively. Then there is a component $\mathcal{C}^{**}_{\ell}$, such that $|\mathcal{C}^{**}_{\ell}|>\frac{k-3}{s+1}n^{**}_{\ell}
\ge ex(n^{**}_{\ell},s,\mathcal{B}_{k-2})$.
If there exists such a $\mathcal{C}^{**}_{\ell}$ satisfying that $\mathcal{C}^{**}_{\ell}\cap \mathcal{C}^*_j\cap \mathcal{C}_i=\emptyset$, then we can find an edge $e'$ containing a vertex $v_1\notin \{u_1, u_2, \ldots, u_{k-1},w_1, w_2, \ldots, u_{k}\}$, and the color of $e'$ is  different  with $\alpha_1,\alpha_2, \ldots, \alpha_{k-1},\beta_1,\beta_2, \ldots, \beta_{k-2}$. Otherwise, $\mathcal{C}^{**}_{\ell}\subseteq \mathcal{C}^*_j$ or $\mathcal{C}^{**}_{\ell}\subseteq \mathcal{C}_i$. We have $n^{**}_{\ell}\ge s+1\ge k$. With the argument similar to the proof of (\ref{eq1}),  we can obtain that
\begin{equation*}\label{eq2}
n^{**}_{\ell}\ge  2k.
\end{equation*}
Thus, there is  a vertex not belonging to  $\{u_1, u_2, \ldots, u_{k-1},w_1, w_2, \ldots, u_{k}\}$. We still denote it by $v_1$. Take an edge $e'$ in  $\mathcal{C}^{**}_{\ell}$ containing $v_1$. Denote this edge by $e'$. So the color of $e'$ is different with $\alpha_1,\alpha_2, \ldots, \alpha_{k-1},\beta_1,\beta_2, \ldots, \beta_{k-2}$. Consider an edge $e''$ containing  $w_1,u_1,v_1$, it must be colored with a color from $\{\alpha_1,\alpha_2, \ldots, \alpha_{k-1}\}$ since otherwise $P$ can be extended. But then $e'\cup e''\cup P'$ is a rainbow $\mathcal{B}_{k}$, a contradiction. Therefore, we have proved the upper bound.
\qed

\subsection{Berge Cycle--Proof of Proposition \ref{BC}}
Note that the  lower bound in Proposition \ref{BC} is obvious, which follows from a similar observation as in Eq.\eqref{H-e} for hypergraphs. Now we prove the upper bound in Proposition \ref{BC}.
Let  $\mathcal{H}$ be a complete $s$-uniform hypergraph on $n$ vertices. By contradiction, consider  a coloring of $\mathcal{H}$ using $ex(n,s,\mathcal{B}_{k-1})+k$ colors  yielding no rainbow $\mathcal{BC}_{k}$. Let $\mathcal{G}$ be a subgraph of $\mathcal{H}$ with $ex(n,s,\mathcal{B}_{k-1})+k$ edges such that each color appears on exactly one edge
of $\mathcal{G}$. So the number of edges of $\mathcal{G}$ is $|\mathcal{G}|=ex(n,s,\mathcal{B}_{k-1})+k>ex(n,s,\mathcal{B}_{k-1})$. Hence there is a rainbow Berge path $P$ of length $k-1$ in $\mathcal{G}$. Denote by $e_1$, $e_2$, \ldots, $e_{k-1}$ the edges of $P$ with colors $\alpha_1$, $\alpha_2$, \ldots, $\alpha_{k-1}$, respectively. And there are $k$ vertices $w_1$, $w_2$, \ldots, $w_{k}$ in $P$ such that $w_i,w_{i+1}\in e_i$ for $i=1,\ldots,k-1$.

 Let $\mathcal{F}$ be the hypergraph obtained from $\mathcal{G}$ by removing all the edges of $P$. Then we have that $|\mathcal{F}|=
ex(n,s,\mathcal{B}_{k-1})+1$.
 Therefore, there is a  Berge path $P^*$ of length $k-1$ in $\mathcal{F}$. Denote by $g_1$, $g_2$, \ldots, $g_{k-1}$ the edges of $P^*$, where
  there are $k$ vertices $z_1$, $z_2$, \ldots, $z_{k}$ in $P^*$ such that $z_i,z_{i+1}\in g_i$ for $i=1,\ldots, k-1$. Consider an $s$-edge $e$ containing $w_1,w_k,z_1,z_k$.
If $e$ is colored with a color not in $\{\alpha_1, \alpha_2, \ldots, \alpha_{k-1}\}$, then $e\cup P$ is a rainbow  $\mathcal{BC}_{k}$. So $e$ is colored with a color belonging to  $\{\alpha_1, \alpha_2, \ldots, \alpha_{k-1}\}$, but now $e\cup P^*$ is a rainbow  $\mathcal{BC}_{k}$, a contradiction. Hence, $ar(n,s,\mathcal{BC}_{k-1})\le ex(n,s,\mathcal{B}_{k-1})+k$ for any possible $n$. This proves Proposition \ref{BC}.
\qed

\vspace{6mm}

\noindent {\bf Acknowledgement.}  Ran Gu was partially supported by
 Natural Science Foundation of Jiangsu Province (No. BK20170860), National Natural Science Foundation of China  (No. 11701143), and the Fundamental Research Funds for the Central Universities. Jiaao Li was partially supported by  National Natural Science Foundation of China (No. 11901318), Natural Science Foundation of Tianjin (No. 19JCQNJC14100) and  the Fundamental Research Funds for the Central Universities, Nankai University (No. 63191425).   Yongtang Shi was partially supported by National Natural Science Foundation of China,
Natural Science Foundation of Tianjin (No. 17JCQNJC00300), the China-Slovenia bilateral project ``Some topics in modern graph theory" (No.~12-6),
Open Project Foundation of Intelligent Information Processing Key Laboratory of Shanxi Province (No. CICIP2018005),
and the Fundamental Research Funds for the Central Universities.

\footnotesize{

}

\end{document}